\documentclass[11pt,leqno]{amsart}
\usepackage[mathcal]{eucal}
\usepackage[utf8]{inputenc}
\usepackage{amsfonts,amssymb}
\usepackage{hyperref}
\usepackage{graphicx}
\usepackage{enumerate,float}
\usepackage{caption}
\usepackage{xfrac}

\usepackage{tikz}
\tikzset{every picture/.style={line width=0.75pt}} 

\numberwithin{equation}{section}

\theoremstyle{plain}
\newtheorem{lemma}{Lemma}[section]
\newtheorem{lemma-defi}[lemma]{Lemma and Definition}
\newtheorem{proposition}[lemma]{Proposition}
\newtheorem{theorem}[lemma]{Theorem}
\newtheorem{corollary}[lemma]{Corollary}
\newtheorem{theoremABC}{Theorem}

\theoremstyle{remark}
\newtheorem{remark}[lemma]{Remark}
\newtheorem*{notation*}{Notation}
\newtheorem{question}[lemma]{Question}

\theoremstyle{definition}
\newtheorem{definition}[lemma]{Definition}
\newtheorem{example}[lemma]{Example}
 \newtheorem{conjecture}[lemma]{Conjecture}
\DeclareMathOperator{\supp}{supp}
\DeclareMathOperator{\dc}{dc}
\DeclareMathOperator{\dens}{\delta}
\DeclareMathOperator{\It}{It}
\DeclareMathOperator{\maxit}{maxi}
\DeclareMathOperator{\minit}{mini}

\makeatletter
\def\moverlay{\mathpalette\mov@rlay}
\def\mov@rlay#1#2{\leavevmode\vtop{%
   \baselineskip\z@skip \lineskiplimit-\maxdimen
   \ialign{\hfil$\m@th#1##$\hfil\cr#2\crcr}}}
\newcommand{\charfusion}[3][\mathord]{
    #1{\ifx#1\mathop\vphantom{#2}\fi
        \mathpalette\mov@rlay{#2\cr#3}
      }
    \ifx#1\mathop\expandafter\displaylimits\fi}
\makeatother
\newcommand{\cupdot}{\charfusion[\mathbin]{\cup}{\cdot}}

\makeatletter
\def\author@andify{%
  \nxandlist {\unskip ,\penalty-1 \space\ignorespaces}%
    {\unskip {} \@@and~}%
    {\unskip \penalty-2 \space \@@and~}%
}
\makeatother
\title[Degree of commutativity of wreath products]{The degree of
  commutativity of wreath products with infinite cyclic top group}

\author[I. de las Heras]{Iker de las Heras} 
\address{Iker de las Heras: Mathematisches Institut, Heinrich-Heine-Universit\"at, 40225
  D\"usseldorf, Germany; Department of Mathematics, University
  of the Basque Country UPV/EHU, 48940 Leioa, Spain}
\email{iker.delasheras@hhu.de}

\author[B. Klopsch]{Benjamin Klopsch} 
\address{Benjamin Klopsch:
  Mathematisches Institut, Heinrich-Heine-Universit\"at, 40225
  D\"usseldorf, Germany} \email{klopsch@math.uni-duesseldorf.de}

\author[A. Zozaya]{Andoni Zozaya} 
\address{Andoni Zozaya: Department of Mathematics, University
  of the Basque Coun\-try UPV/EHU, 48940 Leioa, Spain}
\email{andoni.zozaya@ehu.eus}
\date{}

\makeatletter
\@namedef{subjclassname@2020}{\textup{2020} Mathematics Subject
  Classification}
\makeatother
\thanks{The first author acknowledges support by the Basque
  Government, grant POS\_2021\_2\_0040.
  The third author is supported by Spanish Ministry of Science, Innovation and Universities' grant FPU17/04822.
  The first and third author acknowledge as well support
  by the Basque Government, project IT483-22, and the Spanish
  Government, project PID2020-117281GB-I00, partly with ERDF funds.
  The authors thank Heinrich-Heine-Universit\"{a}t D\"{u}sseldorf,
  where a large part of this research was carried out.}

\begin{document}

\maketitle

\begin{abstract}
  The degree of commutativity of a finite group is the probability
  that two uniformly and randomly chosen elements commute.  This
  notion extends naturally to finitely generated groups~$G$: the
  degree of commutativity $\dc_S(G)$, with respect to a given finite
  generating set~$S$, results from considering the fractions of
  commuting pairs of elements in increasing balls around $1_G$ in the
  Cayley graph $\mathcal{C}(G,S)$.  We focus on restricted wreath
  products of the form $G = H \wr \langle t \rangle$, where $H \ne 1$
  is finitely generated and the top group $\langle t \rangle$ is
  infinite cyclic.  In accordance with a more general conjecture, we
  show that $\dc_S(G) = 0$ for such groups~$G$, regardless of the
  choice of~$S$.

  This extends results of Cox who considered lamplighter groups with
  respect to certain kinds of generating sets.  We also derive a
  generalisation of Cox's main auxiliary result: in `reasonably large'
  homomorphic images of wreath products $G$ as above, the image of the
  base group has density zero, with respect to certain types of
  generating sets.
\end{abstract}

\section{Introduction}

Let $G$ be a finitely generated group, with finite generating set~$S$.
For $n \in \mathbb{N}_0$, let $B_S(n) = B_{G,S}(n)$ denote the ball of radius
$n$ in the Cayley graph $\mathcal{C}(G,S)$ of $G$ with respect to~$S$.
Following Antol\'in, Martino and Ventura~\cite{AnMaVe17}, we define
the \emph{degree of commutativity} of $G$ with respect to $S$ as
\begin{equation*}
  \dc_S(G) = \limsup_{n \rightarrow
    \infty} \frac{\vert\{ (g,h) \in B_S(n) \times B_S(n) \mid
    gh=hg\}\vert}{\vert B_S(n)\vert^2}.
\end{equation*}
We remark that this notion can be viewed as a special instance of a
more general concept, where the degree of commutativity is defined
with respect to `reasonable' sequences of probability measures on~$G$,
as discussed in a preliminary \texttt{arXiv}-version
of~\cite{AnMaVe17} or, in more detail, by Tointon in~\cite{To20}.

If $G$ is finite, the invariant $\dc_S(G)$ does not depend on the
particular choice of~$S$, as the balls stabilise and
$\dc(G) = \dc_S(G)$ simply gives the probability that two uniformly
and randomly chosen elements of~$G$ commute.  This situation was
studied already by Erd\H{o}s and Tur\'an~\cite{ErTu68}, and further
results were obtained by various authors over the years;
for example, see~\cite{GuRo06, Gu73, MaToVaVe21, Ne89,
    Ru79}.  For infinite groups~$G$, it is generally not known
whether $\dc_S(G)$ is independent of the particular choice of~$S$.

The degree of commutativity is naturally linked to the following
concept of density, which is employed, for instance, in~\cite{BuVe02}.
The \emph{density} of a subset $X\subseteq G$ with respect to $S$ is
\[
  \dens_S(X) = \dens_{G,S}(X) = \limsup_{n \rightarrow \infty} \frac{\vert X \cap
    B_S(n)\vert }{\vert B_S(n)\vert }.
\]
If the group $G$ has sub-exponential word growth, then the density
function~$\dens_S$ is bi-invariant; compare
with~\cite[Prop.~2.3]{BuVe02}.  Based on this fact, the
following can be proved, initially for residually finite groups and
then without this additional restriction, even in the more general
context of suitable sequences of probability measures;
see~\cite[Thm.~1.3]{AnMaVe17} and~\cite[Thms.~1.6 and 1.17]{To20}.

\begin{theorem}[\textup{Antol\'in, Martino and
    Ventura~\cite{AnMaVe17}; Tointon~\cite{To20}}]
  \label{thm:AMV}
  Let $G$ be a finitely generated group of sub-exponential word growth,
  and let $S$ be a finite generating set of~$G$.  Then $\dc_S(G) > 0$
  if and only if $G$ is virtually abelian.  Moreover, $\dc_S(G)$ does
  not depend on the particular choice of~$S$.
\end{theorem}

The situation is far less clear for groups of exponential word growth.
In this context, the following conjecture was raised
in~\cite{AnMaVe17}.

\begin{conjecture}[\textup{Antol\'in, Martino and
    Ventura~\cite{AnMaVe17}}]
  \label{conjecture}
  Let $G$ be a finitely generated group of exponential word growth and
  let $S$ be a finite generating set of~$G$.  Then, $\dc_S(G)=0$,
  irrespective of the choice of~$S$.
\end{conjecture}

In~\cite{AnMaVe17} the conjecture was already confirmed for
non-elementary hyperbolic groups, and Valiunas~\cite{Va19} confirmed
it for right-angled Artin groups (and more general graph products of
groups) with respect to certain generating sets. Furthermore,
Cox~\cite{Co18} showed that the conjecture holds, with respect to
\emph{selected} generating sets, for (generalised) lamplighter groups,
that is for restricted standard wreath products of the form
$G = F \wr \langle t \rangle$, where $F \ne 1$ is finite and
$\langle t \rangle$ is an infinite cyclic group.  Wreath products of
such a shape are basic examples of groups of exponential word growth;
in Section~\ref{sec:preliminaries} we briefly recall the wreath
product construction, here we recall that
$G = N \rtimes \langle t \rangle$ with base group
$N = \bigoplus_{i \in \mathbb{Z}} F^{t^i}$.  In the
present paper, we make a significant step forward in two directions,
by confirming Conjecture~\ref{conjecture} for an even wider class of
restricted standard wreath products and with respect to
\emph{arbitrary} generating sets.

\begin{theoremABC}
  \label{thm:main}
  Let $G = H \wr \langle t \rangle$ be the restricted wreath product
  of a finitely generated group~$H \ne 1$ and an infinite cyclic group
  $\langle t \rangle \cong C_\infty$.  Then $G$ has degree of
  commutativity~$\dc_S(G)=0$, for every finite generating set $S$
  of~$G$.
\end{theoremABC}

One of the key ideas in~\cite{Co18} is to reduce the desired
conclusion $\dc_S(G) = 0$, for the wreath products
$G = N \rtimes \langle t \rangle$ under consideration, to
the claim that the base group $N$ has density
$\dens_S(N) =0$ in~$G$.  We proceed in a similar way and
derive Theorem~\ref{thm:main} from the following density result, which
constitutes our main contribution.

\begin{theoremABC}
  \label{thm:main-density}
  Let $G = H \wr \langle t \rangle$ be the restricted wreath product
  of a finitely generated group~$H$ and an infinite cyclic group
  $\langle t \rangle \cong C_\infty$.  Then the base group
  $N = \bigoplus_{i\in \mathbb{Z}} H^{t^i}$ has density
  $\dens_S(N)=0$ in $G$, for every finite generating set
  $S$ of~$G$.
\end{theoremABC}

The limitation in~\cite{Co18} to special generating sets $S$ of
lamplighter groups $G$ is due to the fact that the arguments used
there rely on explicit minimal length expressions for elements
$g\in G$ with respect to~$S$.  If one restricts to generating sets
which allow control over minimal length expressions in a similar, but
somewhat weaker way, it is, in fact, possible to simplify and
generalise the analysis considerably.  In this way we arrive at the
following improvement of the results in~\cite[\textsection 2.2]{Co18}, for
homomorphic images of wreath products.

\begin{theoremABC}
  \label{thm:main-2}
  Let $G$ be a finitely generated group of exponential word growth of
  the form $G= N \rtimes \langle t \rangle$, where
  
  \smallskip
  
  \begin{enumerate}[\rm (a)]
  \item the subgroup $\langle t \rangle$ is infinite cyclic; 
  \item the normal subgroup
    $N = \langle \bigcup \big\{ H^{t^i} \mid i \in
    \mathbb{Z} \big\} \rangle$ is generated by the
    $\langle t \rangle$-conjugates of a finitely generated
    subgroup $H$ of $N$;
  \item the $\langle t \rangle$-conjugates of this group $H$ commute
    elementwise: $\big[H^{t^i}, H^{t^j} \big] = 1$ for all
    $i, j \in \mathbb{Z}$ with $H^{t^i} \ne H^{t^j}$.
  \end{enumerate}
  
  \smallskip
  
  \noindent Suppose further that $S_0$ is a finite generating set for $H$ and
  that the exponential growth rates of $H$ with respect to $S_0$ and
  of $G$ with respect to $S = S_0 \cup \{ t \}$ satisfy
  \begin{equation}
    \label{eq:inequality}
    \lim_{n \rightarrow \infty} \sqrt[n]{\vert B_{H,S_0}(n)\vert } < \lim_{n
      \rightarrow \infty} \sqrt[n]{\vert B_{G,S}(n)\vert }. 
  \end{equation}
  Then $N$ has density $\dens_S(N)=0$ in $G$
  with respect to~$S$.
\end{theoremABC}

For finitely generated groups $G$ of sub-exponential word growth, the
density of a subgroup of infinite index, such as $N$ in
$G = N \rtimes \langle t \rangle$ with
$\langle t \rangle \cong C_\infty$, is always~$0$;
see~\cite{BuVe02}. Thus Theorem~\ref{thm:main-2} has the following
consequence.

\begin{corollary}
  Let $G = A \rtimes \langle t \rangle$ be a finitely generated group,
  where $A$ is abelian and $\langle t \rangle \cong C_\infty$.  Then
  $A$ has density $\dens_S(A) = 0$ in $G$, with respect to any finite
  generating set of $G$ that takes the form $S = S_0 \cup \{t\}$ with
  $S_0 \subseteq A$.
\end{corollary}

Next we give a very simple concrete example to illustrate
that the technical condition~\eqref{eq:inequality} in
Theorem~\ref{thm:main-2} is not redundant: the
  situation truly differs from the one for groups of sub-exponential
  word growth. It is not difficult to craft more complex examples.

\begin{example}
  \label{ex:counterexample}
  Let $G = F \times \langle t\rangle$, where $F = \langle x,y \rangle$
  is the free group on two generators and
  $\langle t \rangle \cong C_\infty$.  Then $F$ has density
  $\dens_S(F) = 1/2 > 0$ in $G$ for the `obvious' generating
  set~$S=\{ x,y,t \}$.

  Indeed, for every $i \in \mathbb{Z}$ we have
  \[
    B_{G,S}(n) \cap F t^i =
    \begin{cases}
      B_{F,\{ x, y\}}(n - \vert i\vert ) t^i & \text{if $n \in \mathbb{N}$ with
        $n \ge \vert i\vert $,} \\
      \varnothing & \text{otherwise,}
    \end{cases}
  \]
  and hence, for all $n \in \mathbb{N}$,
  \[
    \vert B_{G,S}(n) \cap F\vert = \vert B_{F,\{x,y\}}(n)\vert
  \]
  and
  \[ \vert B_{G,S}(n)\vert  = \vert B_{F,\{x,y\}}(n)\vert  + 2 \sum\nolimits_{i=1}^{n}
    \vert B_{F,\{x,y\}}(n-i)\vert .
  \]
  This yields
  \begin{multline*}
    \frac{\vert B_{G,S}(n) \cap F\vert }{\vert B_{G,S}(n)\vert } =\frac{2\cdot 3^n
      -1}{2\cdot 3^n -1 + 2\sum_{i=1}^{n}\left(2\cdot 3^{n-i}
        -1\right)} \\ = \frac{2\cdot 3^n -1}{4 \cdot 3^n -2n-3}
    \to \frac{1}{2} \qquad
    \text{as $n \to \infty$.}
  \end{multline*}
  We remark that in this example $F$ and $G$ have the same exponential
  growth rates:
  \[
    \lim_{n \to \infty} \sqrt[n]{\vert B_{F,\{x,y\}}(n)\vert } = \lim_{n \to
      \infty} \sqrt[n]{\vert B_{G,S}(G)\vert } =3.
  \]
  Furthermore, the argument carries through without any obstacles with
  any finite generating set $S_0$ of $F$ in place of $\{x,y\}$.
\end{example}

Finally, we record an open question that suggests itself rather
naturally.

\begin{question} \label{que:alternative}
  Let $G$ be a finitely generated group such that $\dc_S(G) >0$ with
  respect to a finite generating set~$S$.  Does it follow that there
  exists an abelian subgroup $A \leq G$ such that $\dens_S(A) >0$?
\end{question}

For groups $G$ of sub-exponential word growth the answer is ``yes'',
as one can see by an easy argument from Theorem~\ref{thm:AMV}.  An
affirmative answer for groups of exponential word growth could be a
step towards establishing Conjecture~\ref{conjecture} or provide a
pathway to a possible alternative outcome.  At a heuristic level, an
affirmative answer to Question~\ref{que:alternative} would fit well
with the results in~\cite{Sh18} and~\cite{To20}. \\


\textbf{Notation.}
  Our notation is mostly standard. For a set $X$, we
    denote by $\mathcal{P}(X)$ its power set. For elements $g,h$ of a
  group $G$, we write $g^h=h^{-1}gh$ and $[g,h] = g^{-1} g^h$.  For a
  finite generating set $S$ of $G$, we denote by $l_S(g)$ the length
  of~$g$ with respect to~$S$, i.e., the distance between $g$ and $1$
  in the corresponding Cayley graph~$\mathcal{C}(G,S)$ so
    that
    \[
      B_S(n) = B_{G,S}(n) = \{ g \in G \mid l_S(g) \le n \} \qquad
      \text{for $n \in \mathbb{N}_0$.}
    \]

  Given $a,b \in \mathbb{R}$ and $T \subseteq \mathbb{R}$, we write
  $[a,b]_T = \{ x \in T \mid a \le x \le b \}$; for instance,
  $[-2,\sqrt{2}]_\mathbb{Z} = \{-2,-1,0,1\}$.  We repeatedly compare
  the limiting behaviour of real-valued functions defined on cofinite
  subsets of $\mathbb{N}_0$ which are eventually non-decreasing and
  take positive values.  For this purpose we employ the conventional
  Landau symbols; specifically we write, for functions $f,g$ of the
  described type,
  \begin{align*}
   \text{$f \in o(g)$, or  $g \in \omega(f)$,} %
    & \quad\text{if $\lim_{n \to \infty} \tfrac{f(n)}{g(n)} = 0$,
      equivalently $\lim_{n \to \infty} \tfrac{g(n)}{f(n)} = \infty$.}
  \end{align*}
  As customary, we use suggestive short notation such as, for
  instance, $f \in o(\log n)$ in place of $f \in o(g)$ for
  $g \colon \mathbb{N}_{\ge 2} \to \mathbb{R}$, $n \mapsto \log(n)$. \\

\textbf{Acknowledgement.}
  We thank two independent referees for detailed and valuable
  feedback.  Their comments triggered us to improve the exposition and
  to sort out a number of minor shortcomings.  In particular, this gave
  rise to Proposition~\ref{pro:exists-q}.


\section{Preliminaries}

\label{sec:preliminaries}

In this section, we collect preliminary and auxiliary results.
Furthermore, we briefly recall the wreath product construction and
establish basic notation.

\subsection{Groups of exponential word growth}
We concern ourselves with groups of exponential word
  growth.  These are finitely generated groups $G$ such that for any
  finite generating set $S$ of~$G$, the \emph{exponential growth
  rate}
\begin{equation}
  \label{eq:exponentail-rate}
  \lambda_S(G) = \lim_{n\rightarrow \infty } \sqrt[n]{\vert B_S(n)\vert } = \inf
  \big\{ \sqrt[n]{\vert B_S(n)\vert } \mid n \in \mathbb{N}_0 \big\}
\end{equation}
of $G$ with respect to~$S$ satisfies $\lambda_S(G) > 1$.  
  Since the word growth sequence $\vert B_S(n) \vert$,
  $n \in \mathbb{N}$, is submultiplicative, i.e.,
  \[
    \vert B_S(n + m) \vert \leq \vert B_S(n) \vert \vert B_S(m) \vert
    \qquad \text{for all } n, m \in \mathbb{N},
  \]
  the limit in \eqref{eq:exponentail-rate} exists and is equal to the
  infimum as stated, by Fekete's lemma~\cite[Corollary VI.57]{Ha03}.
We will use the following basic estimates:
\begin{align*}
  \lambda_S(G)^n \le \vert B_S(n)\vert  %
  & \quad \text{for all $n \in \mathbb{N}_0$}, \\
  \intertext{and, for each $\varepsilon \in \mathbb{R}_{>0}$,}
  \vert B_S(n)\vert  \le (\lambda_S(G) + \varepsilon)^n %
  & \quad \text{for all sufficiently large $n \in \mathbb{N}$.}
\end{align*}

In the proof of Theorem~\ref{thm:main-2}, the following
two auxiliary results are used.

\begin{lemma}
  \label{lem:stirling}
  For each $\alpha \in [0,1]_\mathbb{R}$, the sequences
  $\sqrt[n]{\binom{n+ \lceil \alpha n \rceil}{\lceil\alpha n\rceil}}$
  and $\sqrt[n]{\binom{n}{\lceil\alpha n\rceil}}$, $n \in \mathbb{N}$,
  converge, and furthermore
  \[
    \lim_{\alpha\rightarrow
      0^+}\left(\lim_{n\rightarrow\infty}\sqrt[n]{\binom{n+ \lceil
          \alpha n \rceil}{\lceil\alpha n\rceil}}\right) =
    \lim_{\alpha\rightarrow
      0^+}\left(\lim_{n\rightarrow\infty}\sqrt[n]{\binom{n}{\lceil\alpha
          n\rceil}}\right)=1.
  \]
  Consequently, if $f \colon \mathbb{N} \to \mathbb{R}_{>0}$ satisfies
  $f\in o(n)$, then the sequence
  $\binom{n +\lceil f(n)\rceil}{\lceil f(n)\rceil}$,
  $n \in \mathbb{N}$, grows sub-exponentially in~$n$, viz.\
  $\sqrt[n]{\binom{n+\lceil f(n)\rceil}{\lceil f(n)\rceil}} \to 1$ as
  $n \to \infty$.
\end{lemma}

\begin{proof}
  For each $\alpha \in [0,1]_\mathbb{R}$, Stirling's approximation for
  factorials yields
  \begin{align*}
    \binom{n + \lceil \alpha n \rceil}{\lceil\alpha n \rceil}
    &\sim\frac{\sqrt{2\pi (n + \lceil \alpha n \rceil)} \big( (n +
      \lceil \alpha n \rceil)/e \big)^{(n + \lceil \alpha n \rceil)}}
      {\sqrt{2\pi \lceil\alpha n\rceil}(\lceil\alpha n\rceil/e)^{\lceil\alpha n\rceil}
      \sqrt{2\pi n}(n/e)^{n}}\\
    &=\frac{\sqrt{n + \lceil \alpha n \rceil}}{\sqrt{2\pi n \lceil\alpha n \rceil}}
      \cdot \frac{\lceil n + \alpha n \rceil^{\lceil n + \alpha n
      \rceil}}{\lceil\alpha n\rceil^{\lceil \alpha n \rceil} n^n}, 
      \quad \text{as $n \to \infty$,}
  \end{align*}
  i.e., the ratio of the left-hand term to the right-hand term tends
  to~$1$ as $n$ tends to infinity. Moreover, for all  $n\in\mathbb{N}$,
  \[
    \frac{\lceil n + \alpha n \rceil^{\lceil n + \alpha n
        \rceil}}{\lceil\alpha n\rceil^{\lceil \alpha n \rceil} n^n}
    \ge \frac{(n + \alpha n)^{n + \alpha n}}{(\alpha
      n+1)^{\alpha n +1} n^n} = n^{-1} \left( \frac{(1+\alpha)^{1+\alpha}}{(\alpha
        +\sfrac{1}{n})^{(\alpha + \sfrac{1}{n})}}
    \right)^{n}
  \]
  and similarly
  \[
    \frac{\lceil n + \alpha n \rceil^{\lceil n + \alpha n
        \rceil}}{\lceil\alpha n\rceil^{\lceil \alpha n \rceil} n^n}
    \le \frac{(n + \alpha n + 1)^{n + \alpha n + 1}}{(\alpha
      n)^{\alpha n} n^n} = n \left( \frac{(1+\alpha
        +\sfrac{1}{n})^{(1+\alpha + \sfrac{1}{n})}}{\alpha^{\alpha} }
    \right)^{n}.
  \]
  This shows that
  \begin{equation*}
    \lim_{n\rightarrow\infty}\sqrt[n]{\binom{n + \lceil \alpha n \rceil}{\lceil \alpha n
        \rceil}} = \frac{(1+ \alpha)^{1 + \alpha}}{\alpha^\alpha}.
  \end{equation*}
  Since $\lim_{\alpha \to 0^+} \alpha^\alpha = 1$, we conclude that
  \begin{equation*}
    \lim_{\alpha\rightarrow
      0^+}\left(\lim_{n\rightarrow\infty}\sqrt[n]{\binom{n+ \lceil
          \alpha n \rceil}{\lceil\alpha n\rceil}}\right) 
    =1. 
  \end{equation*}
  
  A similar computation yields that the second sequence
  $\sqrt[n]{\binom{n}{\lceil\alpha n\rceil}}$, $n \in \mathbb{N}$,
  converges.  Again directly, or by virtue of
  \begin{equation*}
    1 \leq \sqrt[n]{\binom{n}{ \lceil \alpha n \rceil}} \leq \sqrt[n]{\binom{n + \lceil \alpha n \rceil}{\lceil \alpha n \rceil }},   
  \end{equation*}
  we conclude that also the second limit, for $\alpha \to 0^+$, is
  equal to~$1$.
\end{proof}

\begin{proposition}  \label{pro:exists-q}
  Let $G$ be a finitely generated group of exponential word growth,
  with finite generating set $S$.  Then there exists a non-decreasing
  unbounded function $q \colon \mathbb{N} \to \mathbb{R}_{\ge 0}$ such
  that $q \in o(n)$ and 
  \[
    \frac{\vert B_S(n)\vert }{\vert B_S(n-q(n))\vert } \to \infty
    \qquad \text{as $n \to \infty$}.
  \]
\end{proposition}

\begin{proof}
  We put $\lambda = \lambda_S(G) > 1$ and write
  $\vert B_S(n) \vert = \lambda^{\sum_{i=1}^n b_i}$, with
  $b_i \in \mathbb{R}_{\ge 0}$ for $i \in \mathbb{N}$, so that the
  sequence $\sum_{i=1}^n b_i$, $n \in \mathbb{N}$, is subadditive and
   \[
     \lim_{n \to \infty} \frac{1}{n} \sum\nolimits_{i=1}^n b_i = 1.
   \]
   In this notation, we seek a non-decreasing unbounded function
   $q \colon \mathbb{N} \to \mathbb{R}_{\ge 0}$ such that,
   simultaneously,
  \begin{equation} \label{eq:requirements-q}
    q(n)/n \to 0 \quad \text{and} \quad \sum\nolimits_{i = n - \lfloor
      q(n) \rfloor +1}^n b_i \to \infty \qquad \text{as
      $n \to \infty$}.
  \end{equation}
  
  We show below that for every $m \in \mathbb{N}$,
  \begin{equation} \label{eq:reduction-to-1/m}
    \sum\nolimits_{i=n - \lfloor n/m \rfloor +1}^n b_i \to \infty
    \qquad \text{as $n \to \infty$}.
  \end{equation}
  From this we see that there is an increasing sequence of
  non-positive integers $c(m)$, $m \in \mathbb{N}$, such that, for
  each $m$,
  \[
    c(m) \ge m^2 \quad \text{and} \quad \forall n \in \mathbb{N}_{\ge
      c(m)}: \sum\nolimits_{i=n - \lfloor n/m \rfloor +1}^n b_i \ge m.
  \]
  Setting $q_1(n) = \lfloor n/m \rfloor$ for $n \in \mathbb{N}$ with
  $c(m) \le n < c(m+1)$ and
  \[
    q(n) = \max \{ q_1(k) \mid k \in [1,n]_\mathbb{Z} \},
  \]
  we arrive at a function $q \colon \mathbb{N} \to \mathbb{R}_{\ge 1}$
  meeting the requirements~\eqref{eq:requirements-q}.

  \smallskip
  
  It remains to establish~\eqref{eq:reduction-to-1/m}.  Let
  $m \in \mathbb{N}$ and put
  $\varepsilon = \varepsilon_m = (6m)^{-1} \in \mathbb{R}_{>0}$.  We
  choose $N = N_\varepsilon \in \mathbb{N}$ minimal subject to
  \[
    1 -\varepsilon \le \frac{1}{n} \sum\nolimits_{i=1}^n b_i \le 1+\varepsilon
    \qquad \text{for all $n \in \mathbb{N}_{\ge N}$.}
  \]
  In the following we deal repeatedly with sums of the form
  \[
    \beta(k) = \sum\nolimits_{i=kN+1}^{kN+N} b_i,
  \]
  for $k \in \mathbb{N}$, and using subadditivity, we obtain
  \[
    \beta(k) \le \beta(0) \le (1+\varepsilon) N \qquad \text{for all
      $k \in \mathbb{N}$.}
  \]

  We consider $n \in \mathbb{N}$ with
  $n \ge (1+\varepsilon) \varepsilon^{-1} N \ge N$ and write
  $n = l N + r$ with $l = l_n \in \mathbb{N}$ and
  $r = r_n \in [0,N-1]_\mathbb{Z}$.  Furthermore, we set
  \[
    t = t_n = \frac{\big\vert \big\{ k \in [0,l-1]_\mathbb{Z}
      \;\big\vert\; \beta(k) > \varepsilon N \big\} \big\vert}{l} \in
    [0,1]_\mathbb{R}.
  \]
  From our set-up, we deduce that
  \begin{multline*}
    1 - \varepsilon \le \frac{1}{n} \sum\nolimits_{i=1}^n b_i \le
    \frac{1}{lN} \bigg( \Big( \sum\nolimits_{k=0}^{l-1} \beta(k)
    \Big) + \beta(l) \bigg) \\
    \le \big( t (1+\varepsilon) + (1-t) \varepsilon \big) +
    \frac{1+\varepsilon}{l} \le t + 2\varepsilon,
  \end{multline*}
  hence $t \ge 1-3\varepsilon$ and consequently
  \begin{multline*}
    \big\vert \big\{ k \in [0,l-1]_\mathbb{Z} \mid \beta(k) >
    \varepsilon N \big\} \cap \big\{ k \in [0,l-1]_\mathbb{Z} \mid
    \lceil (1-6\varepsilon)l \rceil + 1 \le k \big\} \big\vert \\ \ge
    t l + \big( l - \lceil (1-6\varepsilon)l \rceil -1 \big) - l \ge
    \big( 1-3\varepsilon - (1-6\varepsilon) \big) l - 2 = 3\varepsilon
    l - 2.
  \end{multline*}
  Since
  \[
    n- \lfloor n/m \rfloor = \lceil (1- 6 \varepsilon)n \rceil \le
    \lceil (1- 6 \varepsilon) (l+1) \rceil N \le (\lceil (1- 6
    \varepsilon) l \rceil + 1 )N,
  \]
  this gives
  \begin{equation*}
    \sum\nolimits_{i=n- \lfloor n/m \rfloor +1}^n b_i 
    \ge \sum\nolimits_{k= \lceil
      (1-6\varepsilon)l \rceil +1}^{l-1} \beta(k)
    \ge
    (3 \varepsilon l - 2)  \varepsilon N, 
  \end{equation*}
  which tends to infinity as $l \to \infty$.  This
  proves~\eqref{eq:reduction-to-1/m}.
\end{proof}

In~\cite[Lemma 2.2]{Pi00} Pittet seems to claim that
\begin{equation*}
  \liminf_{n \to \infty} \frac{\vert B_S(n)\vert }{\vert B_S(n-1)\vert } > 1,
\end{equation*}
from which Proposition~\ref{pro:exists-q} could be derived much more
easily.  However, we found the explanations in \cite{Pi00} not fully
conclusive and thus opted to work out an independent argument.
Naturally, it would be interesting to establish a more effective
version of Proposition~\ref{pro:exists-q}, if possible.  


\subsection{Wreath products}
\label{sec:wreath-products}
We recall that a group $G = H \wr K$ is the restricted standard
\emph{wreath product} of two subgroups $H$ and~$K$, if it decomposes
as a semidirect product $G = N \rtimes K$, where the
normal closure of $H$ is the direct sum
$N = \bigoplus_{k \in K} H^k$ of the various conjugates
of $H$ by elements of~$K$; the groups $N$ and $K$ are
referred to as the \emph{base group} and the \emph{top group} of the
wreath product~$G$, respectively. Since we do not consider complete
standard wreath products or more general types of permutational wreath
products, we shall drop the terms ``restricted'' and ``standard'' from
now on.

Throughout the rest of this section, let
\begin{equation} \label{eq:standing-assumption-G} G = H \wr \langle t
  \rangle = N \rtimes \langle t \rangle \qquad \text{with
    base group} \qquad N = \bigoplus\nolimits_{i \in
    \mathbb{Z}} H^{t^i}
\end{equation}
be the wreath product of a finitely generated subgroup~$H$ and an
infinite cyclic subgroup $\langle t \rangle \cong C_\infty$.  Every
element $g \in G$ can be written uniquely in the form
\[
  g = \widetilde{g} \, t^{\rho(g)},
\]
where $\rho(g) \in \mathbb{Z}$ and
$\widetilde{g} = \prod\nolimits_{i \in \mathbb{Z}} (g_{\vert i})^{\,
  t^i} \in N$ with `coordinates' $g_{\vert i} \in H$. The
support of the product decomposition of $\widetilde{g}$ is finite and
we write
\[
  \supp(g) = \{ i \in \mathbb{Z} \mid g_{\vert i} \neq 1 \}.
\]

Furthermore, it is convenient to fix a finite symmetric generating set
$S$ of~$G$; thus $G = \langle S \rangle$, and $g \in S$
  implies $g^{-1} \in S$.  We put $d = \vert S\vert $ and fix an
ordering of the elements of~$S$:
\begin{equation}\label{eq:standing-assumption-S}
  S = \{ s_1, \ldots, s_d \}, \qquad \text{with
    $s_j = \widetilde{s_{j}} \, t^{\rho(s_j)}$ for $j \in [1,d]_\mathbb{Z}$,}
\end{equation}
where $\widetilde{s_1}, \ldots, \widetilde{s_d} \in N$.
We write
$r_S = \max \big\{ \rho(s_j) \mid j \in [1,d]_\mathbb{Z} \big\}
 \in \mathbb{N}$.

\begin{definition}
  \label{def:itinerary}
  An \emph{$S$-expression} of an element $g \in G$ is (induced by) a
  word $W = \prod_{k=1}^l X_{\iota(k)}$ in the free semigroup
  $\langle X_1, \ldots, X_d \rangle$ such that
  \begin{equation}
    \label{eq:writing}
    g = W(s_1, \ldots, s_d) = \prod\nolimits_{k=1}^l s_{\iota(k)} ;
  \end{equation}
  here $W$ determines and is determined by the function
  $\iota = \iota_W \colon [1,l]_\mathbb{Z} \to [1,d]_\mathbb{Z}$.  In
  this situation the standard process of collecting powers of
    $t$ to the right yields
  \begin{equation} \label{equ:recover-g-from-It} g = \widetilde{g} \,
    t^{-\sigma(l)} \qquad \text{with} \quad \widetilde{g} =
    \prod\nolimits_{k=1}^l \widetilde{s_{\iota(k)}}^{\,
      t^{\sigma(k-1)}},
  \end{equation}
  where $\sigma = \sigma_{S,W}$ is short for the negative\footnote{At
    this stage the sign change is a price we pay for not introducing notation
    for left-conjugation; Example~\ref{exa:itinery-etc} illustrates
    that $\sigma$ plays a convenient role in the concept of itinerary.}
  cumulative exponent function
  \[
    \sigma_{S,W} \colon [0,l]_\mathbb{Z} \to \mathbb{Z}, \quad k
    \mapsto - \sum\nolimits_{j=1}^k \rho \big( s_{\iota(j)} \big).
  \]

  We define the \emph{itinerary} of $g$ associated to the
  $S$-expression~\eqref{eq:writing} as the pair
  \[
    \It(S,W) = (\iota_W,\sigma_{S,W}),
  \]
  and we say that $\It(S,W)$ has length~$l$, viz.\ the length of the
  word~$W$.  For the purpose of concrete calculations it
    is helpful to depict the functions $\iota_{W}$ and $\sigma_{S, W}$
    as finite sequences.  The term `itinerary' refers
  to~\eqref{equ:recover-g-from-It}, which indicates how $g$ can be
  built stepwise from the sequences $\iota_W$ and $\sigma_{S,W}$; see
  Example~\ref{exa:itinery-etc} below.  In particular, $g$
    is uniquely determined by the itinerary $\It(S,W) = (\iota,\sigma)$
    and, accordingly, we refer to $g$ as the element corresponding to
    that itinerary.  But unless $G$ is trivial and $S$ is empty, the
    element $g$ has, of course, infinitely many $S$-expressions which
    in turn give rise to infinitely many distinct itineraries of one
    and the same element.
  
   For discussing features of the exponent function
    $\sigma_{S,W}$, we call
  \begin{equation*}
    \maxit \!\big( \It(S,W) \big) = \max (\sigma_{S,W}) \qquad \text{and} \qquad
    \minit \!\big( \It(S,W) \big) = \min (\sigma_{S,W})
  \end{equation*}
  the \emph{maximal} and \emph{minimal itinerary points}
  of~$\It(S,W)$.  Later we fix a representative function
  $\mathcal{W} \colon G \to \langle X_1, \ldots, X_d \rangle$,
  $g \mapsto W_g$ which yields for each element of $G$ an
  $S$-expression of shortest possible length.  In that situation we
  suppress the reference to $S$ and refer to
  \[
    \It_\mathcal{W}(g) = \It(S,W_g), \; \maxit_\mathcal{W}(g) =
    \maxit \!\big( \It_\mathcal{W}(g) \big), \; \minit_\mathcal{W}(g) =
    \minit \!\big( \It_\mathcal{W}(g) \big)
  \]
  as the $\mathcal{W}$-\emph{itinerary}, the \emph{maximal
    $\mathcal{W}$-itinerary point} and the \emph{minimal
    $\mathcal{W}$-itinerary point} of any given element~$g$.
\end{definition}

To illustrate the terminology we discuss a concrete
  example.

\begin{example}
  \label{exa:itinery-etc}
  A typical example of the wreath products that we consider is the
  lamplighter group
  \[
    G = \langle a, t \mid a^{2} = 1, [a, a^{t^i}]=1 \text{ for
      $i \in \mathbb{Z}$} \rangle = \bigoplus\nolimits_{i \in \mathbb{Z}}
    \langle a_i \rangle \rtimes \langle t \rangle \cong C_2\wr
    C_\infty,
  \]
  where $a_i = a^{\, t^i}$ for each $i \in \mathbb{Z}$.  We consider
  the finite symmetric generating set
  \[
    S = \{ s_1, \ldots, s_5 \}
  \]
  with
  \[
    s_1 = a_4 t^{-3}, \, s_2 = t^{-2}, \, s_3 = s_1^{\, -1} = a_1
    t^3, \, s_4 = s_2^{\, -1} = t^2, \, s_5  = a_0 = s_5^{\, -1}.
  \]

  Let $g =\widetilde{g} \, t^{3}$ be such that $g_{\vert i} = a$ for
  $i \in \{-1,1,2,6\}$ and $g_{\vert i} = 1$ otherwise.  Then we have
  $\rho(g) = 3$, $\supp(g) = \{-1,1,2,6\}$, and
  \begin{equation}
    \label{eq:writing-example}
    g = t^{-2} \cdot a_0 \cdot a_4t^{-3} \cdot \big( t^2
      \big)^{\, 2} \cdot a_0
    \cdot t^{2} \cdot a_0 \cdot t^{2} = s_2 \, s_5 \, s_1
      \, s_4^{\, 2} \, s_5 \, s_4 \, s_5 \, s_4
  \end{equation}
  is an $S$-expression for $g$ of length~$9$, based on
  $W = X_2 X_5 X_1 X_4^{\, 2} X_5 X_4 X_5 X_4$.  The
  itinerary $I = \It(S,W)$ associated to this $S$-expression for $g$
  is
  \begin{equation} \label{eq:itinery-example} I = (\iota, \sigma) =
    \big( (2, 5, 1, 4, 4, 5, 4, 5, 4), (0,2,2,5,3,1,1,-1,-1,-3) \big),
  \end{equation}
  where we have written the maps $\iota$ and $\sigma$ in sequence
  notation.  Furthermore, we see that $\maxit(I) = 5$ and $\minit(I) =
  -3$.  Figure~\ref{fig:itinerary} gives a pictorial description of
  part of the information contained in~$I$.
  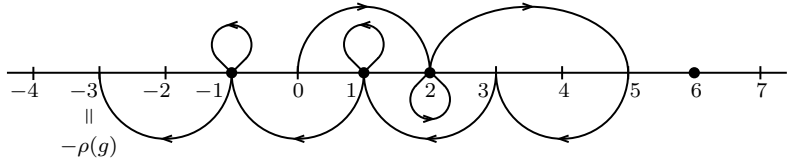
\begin{figure}[H]
    \centering
    \begin{tikzpicture}[x=0.5pt,y=0.5pt,yscale=-1,xscale=1]
      \draw (60,100) -- (650,100) ;
      \draw (80,95.5) -- (80,105.5) ;
      \draw (130,95) -- (130,105) ;
      \draw (180,95) -- (180,105) ;
      \draw (280,95) -- (280,105) ;
      \draw (430,95) -- (430,105) ;
      \draw (480,95) -- (480,105) ;
      \draw (530,95) -- (530,105) ;
      \draw (630,95) -- (630,105) ;
    
      \draw [draw opacity=0] (280,100) .. controls (280,100) and
      (280,100) .. (280,100) .. controls (280,72.39) and (302.39,50)
      .. (330,50) .. controls (357.61,50) and (380,72.39) .. (380,100)
      -- (330,100) -- cycle ; \draw (280,100) .. controls (280,100)
      and (280,100) .. (280,100) .. controls (280,72.39) and
      (302.39,50) .. (330,50) .. controls (357.61,50) and (380,72.39)
      .. (380,100) ;
      \draw [draw opacity=0] (380,100) .. controls (380,100) and
      (380,100) .. (380,100) .. controls (380,72.39) and (413.58,50)
      .. (455,50) .. controls (496.42,50) and (530,72.39) .. (530,100)
      -- (455,100) -- cycle ; \draw (380,100) .. controls (380,100)
      and (380,100) .. (380,100) .. controls (380,72.39) and
      (413.58,50) .. (455,50) .. controls (496.42,50) and (530,72.39)
      .. (530,100) ;
      \draw [draw opacity=0] (530,100) .. controls (530,100) and
      (530,100) .. (530,100) .. controls (530,127.61) and (507.61,150)
      .. (480,150) .. controls (452.39,150) and (430,127.61)
      .. (430,100) -- (480,100) -- cycle ; \draw (530,100) .. controls
      (530,100) and (530,100) .. (530,100) .. controls (530,127.61)
      and (507.61,150) .. (480,150) .. controls (452.39,150) and
      (430,127.61) .. (430,100) ;
      \draw [draw opacity=0] (430,100) .. controls (430,100) and
      (430,100) .. (430,100) .. controls (430,127.61) and (407.61,150)
      .. (380,150) .. controls (352.39,150) and (330,127.61)
      .. (330,100) -- (380,100) -- cycle ; \draw (430,100) .. controls
      (430,100) and (430,100) .. (430,100) .. controls (430,127.61)
      and (407.61,150) .. (380,150) .. controls (352.39,150) and
      (330,127.61) .. (330,100) ;
      \draw [draw opacity=0] (330,100) .. controls (330,100) and
      (330,100) .. (330,100) .. controls (330,127.61) and (307.61,150)
      .. (280,150) .. controls (252.39,150) and (230,127.61)
      .. (230,100) -- (280,100) -- cycle ; \draw (330,100) .. controls
      (330,100) and (330,100) .. (330,100) .. controls (330,127.61)
      and (307.61,150) .. (280,150) .. controls (252.39,150) and
      (230,127.61) .. (230,100) ; \draw (325,47.5) -- (332,50) --
      (325,52.5) ; \draw (450,47.5) -- (457,50) -- (450,52.5) ; \draw
      (485,152.5) -- (477.96,150) -- (484.91,147.5) ; \draw
      (385,152.5) -- (377.96,150) -- (384.91,147.5) ; \draw
      (285,152.5) -- (277.96,150) -- (284.91,147.5) ;

      \draw [fill={rgb, 255:red, 0; green, 0; blue, 0 } ,fill
      opacity=1 ] (226.5,100) .. controls (226.5,98.07) and
      (228.07,96.5) .. (230,96.5) .. controls (231.93,96.5) and
      (233.5,98.07) .. (233.5,100) .. controls (233.5,101.93) and
      (231.93,103.5) .. (230,103.5) .. controls (228.07,103.5) and
      (226.5,101.93) .. (226.5,100) -- cycle ;
      \draw [fill={rgb, 255:red, 0; green, 0; blue, 0 } ,fill
      opacity=1 ] (326.5,100) .. controls (326.5,98.07) and
      (328.07,96.5) .. (330,96.5) .. controls (331.93,96.5) and
      (333.5,98.07) .. (333.5,100) .. controls (333.5,101.93) and
      (331.93,103.5) .. (330,103.5) .. controls (328.07,103.5) and
      (326.5,101.93) .. (326.5,100) -- cycle ;
      \draw [fill={rgb, 255:red, 0; green, 0; blue, 0 } ,fill
      opacity=1 ] (376.5,100) .. controls (376.5,98.07) and
      (378.07,96.5) .. (380,96.5) .. controls (381.93,96.5) and
      (383.5,98.07) .. (383.5,100) .. controls (383.5,101.93) and
      (381.93,103.5) .. (380,103.5) .. controls (378.07,103.5) and
      (376.5,101.93) .. (376.5,100) -- cycle ;

      \draw [fill={rgb, 255:red, 0; green, 0; blue, 0 } ,fill
      opacity=1 ] (576.5,100) .. controls (576.5,98.07) and
      (578.07,96.5) .. (580,96.5) .. controls (581.93,96.5) and
      (583.5,98.07) .. (583.5,100) .. controls (583.5,101.93) and
      (581.93,103.5) .. (580,103.5) .. controls (578.07,103.5) and
      (576.5,101.93) .. (576.5,100) -- cycle ;

      \draw [draw opacity=0] (230,100) .. controls (230,100) and
      (230,100) .. (230,100) .. controls (230,127.61) and (207.61,150)
      .. (180,150) .. controls (152.39,150) and (130,127.61)
      .. (130,100) -- (180,100) -- cycle ; \draw (230,100) .. controls
      (230,100) and (230,100) .. (230,100) .. controls (230,127.61)
      and (207.61,150) .. (180,150) .. controls (152.39,150) and
      (130,127.61) .. (130,100) ; \draw (185,152.5) -- (177.96,150) --
      (184.91,147.5) ;
      \draw (229,99) -- (219,89) ;
      \draw [draw opacity=0] (219.39,89.41) .. controls (216.68,86.69)
      and (215,82.94) .. (215,78.8) .. controls (215,70.52) and
      (221.72,63.8) .. (230,63.8) .. controls (238.28,63.8) and
      (245,70.52) .. (245,78.8) .. controls (245,82.94) and
      (243.32,86.69) .. (240.61,89.41) -- (230,78.8) -- cycle ; \draw
      (219.39,89.41) .. controls (216.68,86.69) and (215,82.94)
      .. (215,78.8) .. controls (215,70.52) and (221.72,63.8)
      .. (230,63.8) .. controls (238.28,63.8) and (245,70.52)
      .. (245,78.8) .. controls (245,82.94) and (243.32,86.69)
      .. (240.61,89.41) ;
      \draw (231,99) -- (241,89) ;
      \draw (329,98.7) -- (319,88.7) ;
      \draw [draw opacity=0] (319.39,89.11) .. controls (316.68,86.39)
      and (315,82.64) .. (315,78.5) .. controls (315,70.22) and
      (321.72,63.5) .. (330,63.5) .. controls (338.28,63.5) and
      (345,70.22) .. (345,78.5) .. controls (345,82.64) and
      (343.32,86.39) .. (340.61,89.11) -- (330,78.5) -- cycle ; \draw
      (319.39,89.11) .. controls (316.68,86.39) and (315,82.64)
      .. (315,78.5) .. controls (315,70.22) and (321.72,63.5)
      .. (330,63.5) .. controls (338.28,63.5) and (345,70.22)
      .. (345,78.5) .. controls (345,82.64) and (343.32,86.39)
      .. (340.61,89.11) ;
      \draw (331,98.7) -- (341,88.7) ;
      \draw (380.94,100) -- (391.05,109.89) ;
      \draw [draw opacity=0] (390.65,109.48) .. controls
      (393.4,112.17) and (395.12,115.9) .. (395.17,120.04) .. controls
      (395.26,128.32) and (388.62,135.12) .. (380.34,135.21)
      .. controls (372.06,135.3) and (365.26,128.66)
      .. (365.17,120.38) .. controls (365.12,116.24) and
      (366.76,112.47) .. (369.44,109.72) -- (380.17,120.21) -- cycle ;
      \draw (390.65,109.48) .. controls (393.4,112.17) and
      (395.12,115.9) .. (395.17,120.04) .. controls (395.26,128.32)
      and (388.62,135.12) .. (380.34,135.21) .. controls
      (372.06,135.3) and (365.26,128.66) .. (365.17,120.38)
      .. controls (365.12,116.24) and (366.76,112.47)
      .. (369.44,109.72) ;
      \draw (378.94,100.02) -- (369.05,110.14) ; \draw (374.89,132.4)
      -- (381.89,134.9) -- (374.89,137.4) ; \draw (335,66.3) --
      (327.96,63.8) -- (334.91,61.3) ; \draw (234.6,66.3) --
      (227.56,63.8) -- (234.51,61.3) ;

      \draw (274,106) node [anchor=north west][inner sep=0.75pt]
      [align=left] {{\scriptsize $0$}};
      \draw (200,106) node [anchor=north west][inner sep=0.75pt]
      [align=left] {{\scriptsize $-1$}};
      \draw (160,106) node [anchor=north west][inner sep=0.75pt]
      [align=left] {{\scriptsize $-2$}};
      \draw (105,106) node [anchor=north west][inner sep=0.75pt]
      [align=left] {{\scriptsize $-3$}};
      \draw (60,106) node [anchor=north west][inner sep=0.75pt]
      [align=left] {{\scriptsize $-4$}};
      \draw (118,142) node [anchor=north west][inner sep=0.75pt]
      [align=left][rotate=90] {{\scriptsize $=$}};
      \draw (98,146) node [anchor=north west][inner sep=0.75pt]
      [align=left] {{\scriptsize $-\rho(g)$}};
      \draw (314,106) node [anchor=north west][inner sep=0.75pt]
      [align=left] {{\scriptsize $1$}};
      \draw (374,106) node [anchor=north west][inner sep=0.75pt]
      [align=left] {{\scriptsize $2$}};
      \draw (414,106) node [anchor=north west][inner sep=0.75pt]
      [align=left] {{\scriptsize $3$}};
      \draw (475,106) node [anchor=north west][inner sep=0.75pt]
      [align=left] {{\scriptsize $4$}};
      \draw (529.07,106) node [anchor=north west][inner sep=0.75pt]
      [align=left] {{\scriptsize $5$}};
      \draw (575,106) node [anchor=north west][inner sep=0.75pt]
      [align=left] {{\scriptsize $6$}};
      \draw (624.5,106) node [anchor=north west][inner sep=0.75pt]
      [align=left] {{\scriptsize $7$}};
    \end{tikzpicture}

    \caption{An illustration of the itinerary of $g$ in
      \eqref{eq:itinery-example} associated to the
      \mbox{$S$-expression} in
      \eqref{eq:writing-example}; the support of
        $\widetilde{g}$ is also indicated}.
    \label{fig:itinerary}
  \end{figure}

An alternative $S$-expression for the same element $g$
    is
    \begin{align}
      \label{eq:writing-example-2}
      g  & =   a_4 t^{-3} \cdot \big( t^2 \big)^2 \cdot  a_0 \cdot a_1 t^3
           \cdot \big( t^{-2} \big)^3  \cdot  a_0 \cdot
           t^{-2} \cdot a_0 \cdot t^{-2} \cdot a_0  \cdot \big( t^2
           \big)^3 \cdot a_0 \cdot a_1 t^{3} \nonumber \\ 
         & = s_1 \, s_4^{\, 2} \, s_5 \, s_3 \, s_2^{\, 3} \, s_5 \,
           s_2 \, s_5 \,  s_2  \, s_5 \, s_4^{\, 3}  \, s_5 \, s_3.
    \end{align}
    It has length $18$ and is based on the semigroup word
    \[
      W' = X_1 \, X_4^{\, 2} \, X_5 \, X_3 \, X_2^{\, 3} \, X_5 \, X_2
      \, X_5 \, X_2 \, X_5 \, X_4^{\, 3} \, X_5 \, X_3.
    \]
    In this case, the itinerary associated to the $S$-expression
    \eqref{eq:writing-example-2} is
    \begin{multline*}
      I' = (\iota', \sigma') = \big( (1, 4, 4, 5 , 3 , 2 , 2 , 2 , 5
      , 2 , 5 , 2 , 5 , 4, 4 , 4, 5, 3),\\
      (0, 3, 1, -1, -1, -4, -2, 0, 2, 2, 4, 4, 6, 6, 4, 2, 0, 0, -3) \big),
    \end{multline*}
    and we observe that $\maxit(I') = 6$ and $\minit(I') = -4.$
\end{example}

There is a natural notion of a product of two
  itineraries, and it has the expected properties.  We collect the
  precise details in a lemma.
   
\begin{lemma-defi}
  \label{lem:itindecomposition}
  In the general set-up described above, suppose that
  $I_1 = (\iota_1, \sigma_1)$ and $I_2 = (\iota_2, \sigma_2)$ are
  itineraries, of lengths $l_1$ and $l_2$, associated to
  $S$-expressions $W_1, W_2$ for elements $g_1, g_2 \in G$.  Then
  $W = W_1W_2$ is an $S$-expression for $g = g_1g_2$; the associated
  itinerary
  \[
    I = \It(S,W) = (\iota,\sigma)
  \]
  has length $l = l_1 + l_2$ and its components are given by
  \begin{align*}
    \iota(k) & =
    \begin{cases}
      \iota_1(k) & \text{if $k \in [1,l_1]_\mathbb{Z}$,} \\
      \iota_2(k-l_1) & \text{if $k \in [l_1+1,l]_\mathbb{Z}$,}
    \end{cases} \\
    \sigma(k) & =
    \begin{cases}
      \sigma_1(k) & \text{if $k \in [0,l_1]_\mathbb{Z}$,} \\
      \sigma_1(l_1) + \sigma_2(k-l_1) & \text{if
        $k \in [l_1+1,l]_\mathbb{Z}$.}
    \end{cases}
  \end{align*}  
  We refer to $I$ as the \emph{product itinerary} and write
  $I = I_1 \ast I_2$.

  Conversely, if $I$ is the itinerary of some element $g\in G$
  associated to some $S$-expression of length~$l$ and if
  $l_1 \in[0,l]_\mathbb{Z}$, there is a unique decomposition
  $I = I_1 \ast I_2$ for itineraries $I_1$ of length $l_1$ and $I_2$
  of length $l_2 = l-l_1$; the corresponding elements $g_1, g_2 \in G$
  satisfy $g = g_1 g_2$.
\end{lemma-defi}
  
\begin{proof}
  The assertions in the first paragraph are easy to verify from 
  \[
    W= W_1 W_2 = \prod\nolimits_{k=1}^{l_1} X_{\iota_1(k)}
    \prod\nolimits_{k=1}^{l_2} X_{\iota_2(k)} =
    \prod\nolimits_{k=1}^{l_1} X_{\iota_1(k)} \prod\nolimits_{k=l_1 +
      1}^{l_1 + l_2} X_{\iota_2(k-l_1)}
  \]
  and the observation that, for $k \in [0,l]_\mathbb{Z}$,
  \begin{multline*}
    \sigma(k) =  - \sum\nolimits_{j=1}^{k} \rho(s_{\iota(k)}) \\
    =
    \begin{cases}
      - \sum_{j=1}^{k} \rho \big(s_{\iota_1(k)} \big) = \sigma_1(k) &
      \text{if $k \le l_1$,} \\
      - \sum_{j=1}^{l_1} \rho(s_{\iota_1(k)}) - \sum_{j=l_1+1}^{k}
      \rho\big(s_{\iota_2(k-l_1)} \big) = \sigma_1(l_1) +
      \sigma_2(k-l_1) & \text{if $k > l_1$.}
    \end{cases}
  \end{multline*}

  Conversely, let $I$ be the itinerary of an element $g$, associated
  to some $S$-expression $W = \prod_{k=1}^l X_{\iota(k)}$ of
  length~$l$, and let $l_1 \in [0,l]_\mathbb{Z}$.  Then $W$ decomposes
  uniquely as a product $W_1 W_2$ of semigroup words of lengths $l_1$
  and $l-l_2$, namely for $W_1 =\prod_{k=1}^{l_1} X_{\iota(k)}$ and
  $W_2 = \prod_{k=l_1+1}^{l} X_{\iota(k)}$.  These are $S$-expressions
  for elements $g_1, g_2$ and $g = g_1 g_2$.  Moreover, $W_1$ and
  $W_2$ give rise to itineraries $I_1, I_2$ such that $I = I_1 * I_2$.
  Since $W_1$ and $I_1$, respectively $W_2$ and $I_2$, determine one
  another uniquely, the decomposition $I = I_1 * I_2$ is unique.
\end{proof}

For a representative function
  $\mathcal{W} \colon G \to \langle X_1, \ldots, X_d \rangle$,
  $g \mapsto W_g$, as discussed in Definition~\ref{def:itinerary}, it
  is typically not the case that $W_{gh} = W_g W_h$ for $g,h \in G$.
  Consequently, it is typically not true that
  $\It_\mathcal{W}(gh) = \It_\mathcal{W}(g) \ast \It_\mathcal{W}(h)$.

\begin{lemma}
  \label{lem:writing}
  Let $G = H \wr \langle t \rangle$ be a wreath product as
  in~\eqref{eq:standing-assumption-G}, with generating set $S$ as
  in~\eqref{eq:standing-assumption-S}. Put
  \[
    C  = C(S) = 1 + \max \big\{ \vert i\vert \mid i \in \supp(s) \text{
        for some } s \in S \big\} \in \mathbb{N}.
  \]
  Then the following hold.

  \smallskip

  \begin{enumerate}[\rm (i)]
  \item \label{enu:lem-1} 
   For every $g \in G$ with itinerary~$I$,
    \[
      \minit(I) - C < \min(\supp(g))  \quad \text{and} \quad
      \max (\supp(g)) < \maxit(I) + C.
    \]
  \item \label{enu:lem-2} Let $u \in H$.  Put
      $m_S = \max\{ C, r_S \} \in \mathbb{N}$ and
    \[
      D = D(S,u) = l_S(u) + 2 \max \left\{ l_S\big(t^j \big) \mid
        0 \le j \le m_S + r_S \right\}  \in \mathbb{N}.
    \]
    Let $g \in G$ with itinerary~$I$, associated to an $S$-expression of
    length~$l_S(g)$.
    Then, for every $j \in \mathbb{Z}$ with $\minit(I) - m_S \le j \le \maxit(I) + m_S$, the elements
    $h = g u^{t^{j+\rho(g)}}, \hbar = u^{t^j} g \in G$ satisfy
    $\rho(h) = \rho(\hbar) = \rho(g)$ and the `coordinates' of
    $h$, $\hbar$ are given by
    \[
      h_{\vert i} =
      \begin{cases}
        g_{\vert i} & \text{if $i \ne j$,} \\
        g_{\vert j} \, u & \text{if $i = j$,}
      \end{cases}
      \qquad
      \hbar_{\vert i} =
      \begin{cases}
        g_{\vert i} & \text{if $i \ne j$,} \\
        u \, g_{\vert j} & \text{if $i = j$}
      \end{cases}
      \qquad \text{for $i \in \mathbb{Z}$}.
    \]
    Furthermore,
    \begin{equation*} 
      l_S(h) \le l_S(g) + D \qquad \text{and} \qquad l_S(\hbar) \le l_S(g) + D.
    \end{equation*}
  \end{enumerate}
\end{lemma}

\begin{proof}
  We write $I = (\iota,\sigma)$ for the given itinerary
    of~$g$, and $l$ denotes the length of~$I$.

  \smallskip
  
  \eqref{enu:lem-1} From \eqref{equ:recover-g-from-It} it follows that
  \begin{align*}
    \supp(g) & \subseteq \bigcup_{1 \le k \le l} \left( \sigma(k-1) +
               \supp(s_{\iota(k)}) \right) \\
    & \subseteq \bigcup_{1 \le k \le l}
    [\sigma(k-1)-C+1,\sigma(k-1)+C-1]_\mathbb{Z};
  \end{align*}
  from this inclusion the two inequalities follow readily.

  \smallskip
  
  \eqref{enu:lem-2} In addition we now have
    $l = l_S(g)$.  The first assertions are very easy to verify.  We
  justify the upper bound for $l_S(h)$; the bound for $l_S(\hbar)$
  is derived similarly.
  
  Since $\minit(I) - m_S \le j \le \maxit(I) + m_S$ and
  since itineraries proceed, in the second coordinate, by steps of
  length at most~$r_S \le m_S$, there exists
  $k \in [0,l]_\mathbb{Z}$ such that
  $\vert j - \sigma(k)\vert \leq m_S$.  If
    $\vert j - \sigma(l)\vert \leq m_S$ we put $k = l-1$; otherwise
  we choose $k$ maximal with $\vert j - \sigma(k)\vert \leq m_S$.
  Next we decompose the itinerary $I$ as the product
  $I = I_1 \ast I_2$ of itineraries $I_1$ of length
  $l_1 = k+1$ and $I_2$ of length
  $l_2 = l-k-1$; compare with
    Lemma~\ref{lem:itindecomposition}.

  We denote by $g_1 = \widetilde{g_1} t^{-\sigma(k+1)}$
    and $g_2 = \widetilde{g_2} t^{\sigma(k+1)+\rho(g)}$ the elements
    corresponding to $I_1$ and~$I_2$ so that
    $g = g_1 g_2 = \widetilde{g_1} \widetilde{g_2}^{t^{\sigma(k+1)}}
    t^{\rho(g)}$.  Moreover, we observe from
    $\vert j - \sigma(k+1) \vert \le m_S+r_S$ that
    \[
     g_3 = u^{t^{j-\sigma(k+1)}} = t^{-j+\sigma(k+1)} \,
        u \, t^{j-\sigma(k+1)}
    \]
    has
    length~$l_3 \le l_S(u) + 2\, l_S \big( t^{j-\sigma(k+1)} \big) \le
    D$.  Our choice of $k$ guarantees that the support of
    $\widetilde{g_2}^{\, t^{\sigma(k+1)}}$ does not overlap with
    $\{j\} = \supp(u^{t^j})$; compare with~\eqref{enu:lem-1}.  Thus
    $\widetilde{g_2}^{\, t^{\sigma(k+1)}}$ and $u^{t^{j}}$, both in
    the base group, commute with one another.  This gives
    \begin{multline*}
      h = g u^{t^{j+\rho(g)}} = \widetilde{g_1}
      \widetilde{g_2}^{t^{\sigma(k+1)}} u^{t^j} t^{\rho(g)} =
      \widetilde{g_1} u^{t^j} \widetilde{g_2}^{t^{\sigma(k+1)}}
      t^{\rho(g)} \\ = g_1 t^{-j + \sigma(k+1)} u t^{j - \sigma(k+1)} g_2
      = g_1 g_3 g_2,
    \end{multline*}
    and we conclude that
    $l_S(h) \le l_1 + l_2 + l_3 \leq l +D = l_S(g) + D$.
\end{proof}


\section{Proofs of Theorems~\ref{thm:main}
  and~\ref{thm:main-density}}

\label{sec:proof-of-main-theorem}

First we explain how Theorem~\ref{thm:main} follows from
Theorem~\ref{thm:main-density}.  The first ingredient is the following
general lemma.

\begin{lemma}[\textup{Antol\'in, Martino and
    Ventura~\cite[Lem.~3.1]{AnMaVe17}}]
  \label{lem:two steps}
  Let $G = \langle S \rangle$ be a group, with finite generating
  set~$S$.  Suppose that there exists a subset $X \subseteq G$
  satisfying
  
  \smallskip
  
  \begin{enumerate}[\rm (a)]
  \item $\dens_S(X) = 0$;
  \item
    $\sup \left\{ \frac{\vert C_G(g) \cap B_S(n)\vert }{\vert B_S(n)\vert } \mid g \in G
      \smallsetminus X \right\} \to 0$ as $n \to \infty$.
  \end{enumerate}
  
  \smallskip
  
  \noindent Then $G$ has degree of commutativity $\dc_S(G)=0$.
\end{lemma}

The second ingredient comes from~\cite[\textsection 2.1]{Co18}, where
Cox shows the following.  If $G = H \wr \langle t \rangle$ is the
wreath product of a finitely generated group $H \ne 1$ and an infinite
cyclic group $\langle t \rangle$, with base group~$N$,
and if $S$ is any finite generating set for~$G,$ then
\[
  \sup \left\{ \tfrac{\vert C_G(g) \cap B_S(n)\vert }{\vert
      B_S(n)\vert } \mid g \in G \smallsetminus N
  \right\} \to 0 \qquad \text{as} \quad n \to \infty.
\]

The idea behind Cox' proof is as follows.  For
  $g \in G \smallsetminus N$, the centraliser $C_G(g)$ is cyclic and
  the `translation length' of $g$ with respect to~$S$ is uniformly
  bounded away from~$0$.  The latter means that there exists
  $\tau_S > 0$ such that
\[
  \inf_{n \in \mathbb{N}} \left\{ \tfrac{ l_S(g^n) }{ n } \mid g \in G
    \smallsetminus N \right\} \geq \tau_S.
\]
Consequently, for $g \in G \smallsetminus N$ the function
$n \mapsto \vert C_G(g) \cap B_S(n) \vert $ is bounded uniformly by a
linear function, while $G$ has exponential word growth.

Thus, Theorem~\ref{thm:main-density} implies Theorem~\ref{thm:main},
and it remains to establish Theorem~\ref{thm:main-density}.
Throughout the rest of this section, let
\begin{equation*}
  G = H \wr \langle t \rangle = N \rtimes \langle t \rangle \qquad
  \text{with base group} \qquad N = \bigoplus\nolimits_{i
    \in \mathbb{Z}} H^{t^i} 
\end{equation*}
be the wreath product of a finitely generated subgroup~$H$ and an
infinite cyclic subgroup $\langle t \rangle$, just as
in~\eqref{eq:standing-assumption-G}.  The exceptional case $H = 1$ poses no obstacle, hence we assume $H \ne 1$.  We fix a finite
symmetric generating set $S = \{s_1, \ldots, s_d\}$ for~$G$ and employ
the notation established around~\eqref{eq:standing-assumption-S}.
Finally, we recall that $G$ has exponential word growth and we write
\[
  \lambda = \lambda_S(G) > 1
\]
for the exponential growth rate of $G$ with respect to~$S$;
see~\eqref{eq:exponentail-rate}.

We start by showing that the subset of $N$ consisting of all elements
with suitably bounded support is negligible in the computation of the
density of~$N$. 

\begin{proposition}
  \label{prop:small-support}
  Fix a representative function $\mathcal{W}$ which yields for each
  element of $G$ an $S$-expression of shortest possible length and let
  $q \colon \mathbb{N} \to \mathbb{R}_{\ge 1}$ be a non-decreasing
  unbounded function such that $q \in o(\log n)$.

Then the sequence of sets
  \begin{equation*}
    R_q(n) = R_{\mathcal{W},q}(n) = \{g \in N \cap B_S(n) \mid
    \maxit_\mathcal{W}(g) - \minit_\mathcal{W}(g) \le q(n) \},
  \end{equation*}
  indexed by $n \in \mathbb{N}$, satisfies
  \begin{equation*}
    \lim_{n \to \infty} \frac{\vert R_{q}(n)\vert }{\vert B_S(n)\vert } = 0.
  \end{equation*}
\end{proposition}

The proof of Proposition~\ref{prop:small-support} is
  preceded by some preparations and two auxiliary lemmata.  We keep in
  place the set-up from Proposition~\ref{prop:small-support}.  For
$i \in \mathbb{Z}$, we write $H_i = H^{t^i}$.  Using the notation
established in Section~\ref{sec:wreath-products}, we accumulate the
`coordinates' of elements of $S$ in a set
\[
  S_0 = \{ s_{\vert i} \mid s \in S, i \in \mathbb{Z} \} = \left\{
    (s_j)_{\vert i} \mid 1 \le j \le d \text{ and } i \in \mathbb{Z}
  \right\} \subseteq H = H_0,
\]
we set $S_i = S_0^{\, t^i} \subseteq H_i$ for $i \in \mathbb{Z}$.
Then $S_i$ is a finite symmetric generating set of~$H_i$ for each
$i \in \mathbb{Z}$. Indeed, every element $h \in H$
  satisfies $h = \widetilde{h} = h_{\vert 0}$ and can thus be written
  in the form
  \[
    h = \left( \prod\nolimits_{k=1}^l \widetilde{s_{\iota(k)}}^{\,
      t^{\sigma(k-1)}} \right) {\vert_0} = \prod\nolimits_{k=1}^{l}
    \big( \widetilde{s_{\iota(k)}}_{\vert \, -\sigma(k-1)} \big),
  \]
  based upon a suitable itinerary $I = (\iota,\sigma)$ of length~$l$.
  We conclude that $H = \langle S_0 \rangle$ and consequently
  $H_i = \langle S_i \rangle$ for $i \in \mathbb{Z}$; the generating
  systems inherit from $S$ the property of being symmetric.

Moreover, we have
$\vert B_{H_i,S_i}(n)\vert = \vert B_{H,S_0}(n)\vert $ for all
$n \in \mathbb{N}$; consequently,
\[
  \lambda_{S_0}(H) = \lambda_{S_i}(H_i) \qquad \text{for all
    $i \in \mathbb{Z}$.}
\]

It is convenient to split the analysis of the
  set~$R_q(n)$ from Proposition~\ref{prop:small-support} into two
  parts.  First we take care of elements whose `coordinates' fall
  within sufficiently small balls around $1$ in~$H$, with respect to
  the generating set~$S_0$.

\begin{lemma}
  \label{lem:small-coordinates}
  In addition to the set-up above, let
  $f \colon \mathbb{N} \rightarrow \mathbb{R}_{> 0}$ be a
  non-decreasing unbounded function such that $f \in o(n/q(n))$.

  Then the sequence of subsets
  \[
    R_q^f(n) = R_{\mathcal{W},q}^f(n) =\big\{ g \in R_q(n) \mid
    g_{\vert i} \in B_{H,S_0}(f(n)) \text{ for all } i \in \mathbb{Z}
    \big\} \subseteq R_q(n),
  \]
  indexed by $n \in \mathbb{N}$, satisfies
  \[
    \lim_{n \rightarrow \infty} \frac{\vert R_q^f(n) \vert}{\vert
      B_S(n) \vert } = 0.
  \]
\end{lemma}

\begin{proof}  
 Let $C = C(S) \in \mathbb{N}$ be as is in
  Lemma~\ref{lem:writing}\eqref{enu:lem-1} and choose
  $C' \in \mathbb{N}$ such that $\lambda^{C'} > \lambda_{S_0}(H)$.
  Then we have, for all sufficiently large $n \in \mathbb{N}$,
  \[
    \big\vert R_q^f(n) \big\vert \leq \big\vert B_{H,S_0}(f(n))
    \big\vert^{2q(n) +2C} \le \lambda^{2C'q(n)f(n) + 2C'Cf(n)} \le
    \lambda^{4C'C q(n) f(n)}.
  \]
  From $f\in o(n/q(n))$ we obtain
  \[
    4C'C q(n) f(n) - n \;\to\; -\infty \qquad \text{as $n \to \infty$}
  \]
  and hence
  \[
    \frac{\vert R_q^f(n)\vert }{\vert B_S(n)\vert } \le \lambda^{4C'C q(n) f(n) - n}
    \;\to\; 0 \qquad \text{as $n \to \infty$.} 
  \] 
\end{proof} 

Next we consider $R_q(n) \smallsetminus R_q^f(n)$, for a
  function $f$ as in Lemma~\ref{lem:small-coordinates} and
  $n \in \mathbb{N}$. For every
$g \in R_q(n) \smallsetminus R_q^f(n),$ we pick $i(g) \in \mathbb{Z}$
with
\[
  \minit_\mathcal{W}(g) - C < i(g) < \maxit_\mathcal{W}(g) +C \qquad
  \text{and} \qquad g_{\vert i(g)} \notin B_{S_0}(f(n)),
\]
where $C = C(S) \in \mathbb{N}$ continues to denote
  the constant from Lemma~\ref{lem:writing}\eqref{enu:lem-1}.  Let
$I = (\iota,\sigma)$, viz.\ $I_g = (\iota_g,\sigma_g)$, denote the
$\mathcal{W}$-itinerary of~$g$.  Then
\[
  g_{\vert i(g)} = \prod\nolimits_{k=1}^{l_S(g)} (s_{\iota(k)})_{\vert
    \, i(g)-\sigma(k-1)}.
\]
By successively cancelling sub-products of adjacent factors that
evaluate to~$1$ and have maximal length with this property (in an
orderly fashion, from left to right, say), we arrive at a `reduced'
product expression
\begin{equation}
  \label{eq:tower}
  g_{\vert  i(g)} = \prod_{j=1}^{\ell}
  (s_{\iota(\kappa(j))})_{\vert \, i(g)-\sigma(\kappa(j)-1)},
\end{equation}
for some $\ell = \ell_g \in [1,l_S(g)]_\mathbb{Z}$ and an increasing
function
$\kappa = \kappa_g \colon [1,\ell]_\mathbb{Z} \to [1,
l_S(g)]_\mathbb{Z}$ that picks out a subsequence of factors.  In
particular, this means that, for $j_1, j_2 \in [1,\ell]_\mathbb{Z}$
with $j_1 < j_2$,
\begin{equation}
  \label{eq:in-particular-kappa}
  \begin{split}
    \prod_{k=\kappa(j_1)+1}^{\kappa(j_2)} (s_{\iota(k)})_{\vert \,
      i(g)-\sigma(k-1)} %
    & = \prod_{j = j_1+1}^{j_2} \; \prod_{k=\kappa(j-1)+1}^{\kappa(j)}
    (s_{\iota(k)})_{\vert \, i(g)-\sigma(k-1)} \\
    & = \prod_{j = j_1+1}^{j_2} (s_{\iota(\kappa(j))})_{\vert \,
      i(g)-\sigma(\kappa(j)-1)} \ne 1,
  \end{split}
\end{equation}
and moreover we have
$l_S(g) \ge \ell \geq l_{S_0}(g_{\vert i(g)}) \geq f(n).$

By means of suitable perturbations, we aim to produce
  from $g$ a collection of $\ell$ distinct elements
  $\dot g(1), \ldots, \dot g(\ell)$ which each carry sufficient
  information to `recover' the initial element~$g$. We proceed as
  follows. For each choice of $j \in [1,\ell]_\mathbb{Z}$ we
decompose the itinerary $I$ for $g$ into a product
$I = I_{j,1} \ast I_{j,2}$ of itineraries of length $\kappa(j)$ and
$l_S(g) - \kappa(j)$; compare with Lemma~\ref{lem:itindecomposition}.
Then $g = g_{j,1} g_{j,2}$, where
$g_{j,1}, g_{j,2}$ denote the elements of $G$
corresponding to $I_{j,1}, I_{j,2}$.  From $g \in R_q(n)$ it follows
that $\maxit(I_{j,1}) - \minit(I_{j,1})$ and
$\maxit(I_{j,2}) - \minit(I_{j,2})$ are bounded by~$q(n)$; in
particular, $\rho(g_{j,1}) \in [-q(n),q(n)]_\mathbb{Z}$.

We define
\begin{equation} \label{eq:gdot} \dot g(j) = g_{j,1} \,
  t^{-3q(n)-4C} \, g_{j,2}
\end{equation}
with $C=C(S)$ as above; see Figure~\ref{fig:hi} for a
pictorial illustration, which features an additional parameter $\tau$
that we introduce in the proof of Lemma~\ref{lem:big-coordinates}.

\begin{figure}[H]
  \centering \hspace{-12pt}
  \begin{tikzpicture}[x=0.70pt,y=0.70pt,yscale=-1,xscale=1]

    \draw (100,100) -- (500,100) ;
    \draw (150,95) -- (150,105) ;
    \draw (200.72,95.13) .. controls (200.74,90.46) and (198.42,88.12)
    .. (193.75,88.1) -- (169.85,88) .. controls (163.18,87.97) and
    (159.86,85.63) .. (159.88,80.96) .. controls (159.86,85.63) and
    (156.52,87.95) .. (149.85,87.92)(152.85,87.94) -- (125.96,87.83)
    .. controls (121.29,87.81) and (118.95,90.13) .. (118.93,94.8) ;
    \draw (371,94.86) .. controls (371.02,90.19) and (368.7,87.85)
    .. (364.03,87.83) -- (340.14,87.73) .. controls (333.47,87.7) and
    (330.15,85.36) .. (330.17,80.69) .. controls (330.15,85.36) and
    (326.81,87.68) .. (320.14,87.65)(323.14,87.66) -- (296.25,87.56)
    .. controls (291.58,87.54) and (289.24,89.86) .. (289.22,94.53) ;
    \draw (203,83.61) -- (287,83.61) ; \draw [shift={(289,83.61)},
    rotate = 180] [fill={rgb, 255:red, 0; green, 0; blue, 0 } ][line
    width=0.08] [draw opacity=0] (9.6,-2.4) -- (0,0) -- (9.6,2.4) --
    cycle ; \draw [shift={(201,83.61)}, rotate = 0] [fill={rgb,
      255:red, 0; green, 0; blue, 0 } ][line width=0.08] [draw
    opacity=0] (9.6,-2.4) -- (0,0) -- (9.6,2.4) -- cycle ;
    \draw (201,78.61) -- (201,88.61) ;
    \draw (289,78.61) -- (289,88.61) ; \draw (289,97.5) -- (289,102.5)
    ;
    \draw (201,97.5) -- (201,102.5) ;
    \draw (119,97.5) -- (119,102.5) ;
    \draw (371.5,97.5) -- (371.5,102.5) ;

    \draw (146,105) node [anchor=north west][inner sep=0.75pt]
    [align=left] {{\scriptsize 0}};
    \draw (105,65) node [anchor=north west][inner sep=0.75pt]
   [font=\tiny] [align=left] {$\displaystyle \ \supp(g_{j,1})$};
    \draw (315,65) node [anchor=north west][inner sep=0.75pt]
    [font=\tiny] 
    [align=left] {$\displaystyle \supp(g_{j,2})$,
      shifted by $\displaystyle -\rho(g_{j,1})+3q(n)+4C$};
    \draw (239.11,70) node [anchor=north west][inner sep=0.75pt]
    [font=\scriptsize] [align=left] {$\displaystyle \tau$};
  \end{tikzpicture}
  \caption{An illustration of the factorisation
    $\dot g(j) = g_{j,1} \, t^{-3q(n)-4C} \,
    g_{j,2}$.}
  \label{fig:hi}
\end{figure}
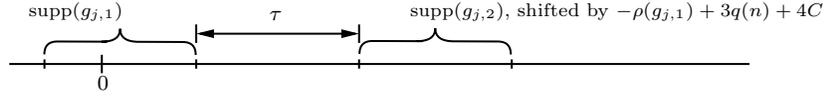

\begin{lemma}
  \label{lem:big-coordinates} In the set-up above, the
    elements $\dot{g}(1), \ldots, \dot{g}(\ell)$ defined in
    \eqref{eq:gdot} satisfy the following:
  \begin{enumerate}[\rm (i)]
  \item \label{enu:ball} for each $j \in [1, \ell]_{\mathbb{Z}}$ the
    element $\dot g(j)$ lies in $B_S( n + (3 q(n) + 4C)l_S(t))$;
  \item \label{enu:recover} for each $j \in [1,\ell]_\mathbb{Z}$ the
    original element $g$ can be recovered from $\dot g(j)$;
  \item \label{enu:pw-distinct} the elements
    $\dot g(1), \ldots, \dot g(\ell)$ are pairwise distinct.
  \end{enumerate}
\end{lemma}

\begin{proof}
  \eqref{enu:ball} Lemma~\ref{lem:itindecomposition} gives
    $l_S(g_{j,1}) + l_S(g_{j,2}) = \ell \le l_S(g) \leq n$, and it is
    clear that $l_S\big(t^{-3q(n)-4C} \big) \leq (3q(n) + 4C)l_S(t)$.

  \smallskip

  \eqref{enu:recover} Let $j \in [1,\ell]_\mathbb{Z}$, and write
  $\mathcal{G}_1 = \supp(g_{j,1})$,
  $\mathcal{G}_2 = \supp(g_{j,2})$.
  Lemma~\ref{lem:writing}\eqref{enu:lem-1} implies that the sets
  $\mathcal{G}_1$ and
  $\mathcal{G}_2 - \rho(g_{j,1}) = \supp \big(
    t^{\rho(g_{j,1})} g_{j,2} \big)$ lie wholly within the interval
  $[-q(n)-C,q(n)+C]_\mathbb{Z}$, hence
  \begin{equation}
    \label{eq:supp-H1-H2}
    \supp\! \big(\dot g(j) \big) = \mathcal{G}_1 \cupdot
    \big( \mathcal{G}_2 - \rho(g_{j,1}) + 3q(n)+4C \big)
  \end{equation}
  with a gap 
  \[
    \tau = \underbrace{\min\!\big( \mathcal{G}_2 - \rho(g_{j,1}) +
      3q(n)+4C \big)}_{\ge\, -q(n)-C+3q(n)+4C \, = \, 2q(n)+3C}
    - \underbrace{\max\!\big( \mathcal{G}_1 \big)}_{\le\,
      q(n)+C} \ge q(n) + 2C, 
  \]
  subject to the standard conventions $\min \varnothing = +\infty$ and
  $\max \varnothing = -\infty$ in special circumstances; see
  Figure~\ref{fig:hi} for a pictorial illustration.

  In contrast, gaps between two elements in
  $\mathcal{G}_1$ or two elements in
  $\mathcal{G}_2$ are strictly less
  than~$q(n)+2C \le \tau$.  Consequently, we can identify
  the two components in~\eqref{eq:supp-H1-H2} and thus
  $\mathcal{G}_1$ and
  $\mathcal{G}_2 - \rho(g_{j,1})$, without any prior
  knowledge of~$j$ or $g_{j,1}, g_{j,2}$.  Therefore, for
  each $i \in \mathbb{Z}$ the $i$th coordinate of $g$ satisfies
  \[
    g_{\vert i} =
    \begin{cases}
      \dot g(j)_{\vert i} \, \dot g(j)_{\vert \, i+3q(n)+4C} & \text{if
        $i \in [-q(n)-C,q(n)+C]$,} \\
      1 & \text{otherwise,}
    \end{cases}
  \]
  and hence $g$ can be recovered from~$\dot g(j)$.

  \smallskip
  
  \eqref{enu:pw-distinct} For $j_1, j_2 \in [1,\ell]_\mathbb{Z}$ with
  $j_1< j_2$ we conclude from our choice of the `reduced' product
  expression~\eqref{eq:tower} and its
  consequence~\eqref{eq:in-particular-kappa} that 
    \begin{multline*}
      \dot g(j_1)_{\vert i(g)} = \big( g_{j_1,1} \big)_{\vert i(g)}=
      \prod\nolimits_{k=1}^{\kappa(j_1)}
      \big( s_{\iota(k)} \big)_{\vert \, i(g)-\sigma(k-1)} \\
      \ne \prod\nolimits_{k=1}^{\kappa(j_2)} \big( s_{\iota(k)}
      \big)_{\vert \, i(g)-\sigma(k-1)} = \big( g_{j_2,1} \big)_{\vert
        i(g)} = \dot g(j_2)_{\vert i(g)}
  \end{multline*}
and hence $\dot g(j_1) \ne \dot g(j_2)$.
\end{proof}

For the proof of Proposition~\ref{prop:small-support} we
  now make a more careful choice of the non-decreasing unbounded
  function $f \colon \mathbb{N} \to \mathbb{R}_{>0}$, which entered
  the stage in Lemma~\ref{lem:small-coordinates}: we arrange that
  \[
    f \in o \big( n/q(n) \big) \quad \text{and} \quad f \in \omega
    \big( (\lambda+1)^{m(n)} \big) \quad \text{for
      $m(n) = \big( 3 q(n) + 4 C \big)\hspace{1pt}l_S(t)$},
  \]
  with $C = C(S)$ as in Lemma~\ref{lem:writing}\eqref{enu:lem-1}.  For
  instance, we can take $f = f_\alpha$ for any real parameter $\alpha$
  with $0 < \alpha < 1$, where
  $f_\alpha(n)=\max\left\{k^{\alpha}/q(k)\mid
    k\in[1,n]_{\mathbb{Z}}\right\}$ for $n\in\mathbb{N}$.
  Indeed, since $q(n) \in o(\log{n})$ and $q(n) \geq 1$
    for all $n \in \mathbb{N}$, each of these functions satisfies
    \[
      \lim_{n \rightarrow \infty} \frac{f_\alpha(n) q(n)}{n} \leq
      \lim_{n \rightarrow \infty} \frac{n^{\alpha} q(n)}{n} = 0.
    \]
    Furthermore, $q(n) \in o(\log n)$ implies
    $q(n)a^{q(n)} \in o(n^\beta)$ for all $a \in \mathbb{R}_{> 1}$ and
    $\beta \in \mathbb{R}_{> 0}$ so that
    \begin{align*}
      \lim_{n \rightarrow \infty} \frac{(\lambda +
      1)^{m(n)}}{f_\alpha(n)} %
      & \le \lim_{n \rightarrow \infty} \frac{q(n)(\lambda +
        1)^{m(n)}}{n^{\alpha}}  \\
      & = (\lambda +1)^{4 C \, l_S(t)} \lim_{n \rightarrow \infty}
        \frac{q(n)(\lambda +1)^{3 l_S(t) q(n)}}{n^\alpha} \\ 
      &   = 0.
    \end{align*}

\begin{proof}[Proof of Proposition \ref{prop:small-support}.]
We continue with the set-up established above; in
    particular, we make use of the refined choice of $f$.  In view
  of~Lemma \ref{lem:small-coordinates} it remains to show that
  \begin{equation*}
    \frac{\vert R_q(n) \smallsetminus R_q^f(n)\vert }{\vert
      B_S(n)\vert } \;\to\; 0 \qquad \text{as $n \to \infty$.}
  \end{equation*}
  We define a map
  \begin{align*}
    F_{n} \colon R_q(n) \smallsetminus R_q^f(n) & \,\rightarrow\, \mathcal{P}
                                                  \big(B_S(n+m(n)) \big) \\
    g & \,\mapsto\, \left\{ \dot g(j) \mid 1 \le j \le \ell_g \right\};
  \end{align*}
see~\eqref{eq:gdot} and
    Lemma~\ref{lem:big-coordinates}\eqref{enu:ball}.  From
  Lemma~\ref{lem:big-coordinates}\eqref{enu:recover} we deduce that
  $F_{n}(g_1) \cap F_{n}(g_2) = \varnothing$ for all
  $g_1, g_2 \in R_q(n) \smallsetminus R_q^f(n)$ with $g_1 \ne g_2$.
  In addition, from $\ell_g \geq f(n)$ and
  Lemma~\ref{lem:big-coordinates}\eqref{enu:pw-distinct} we deduce
  that $\vert F_{n}(g)\vert \geq f(n)$ for all
  $g \in R_q(n) \smallsetminus R_q^f(n)$.  This yields
  \[
    \big\vert  B_S(n+m(n)) \big\vert  \geq f(n) \, \big\vert  R_q(n)
    \smallsetminus R_q^f(n) \big\vert ,
  \]
  and hence, by submultiplicativity,
  \begin{equation*}
  \begin{split}
    \frac{\vert R_q(n) \smallsetminus R_q^f(n)\vert }{\vert B_S(n)\vert } \le
    \frac{\vert B_S(n+m(n))\vert }{f(n) \, \vert B_S(n)\vert } &\le \frac{\vert B_S(m(n))\vert }{f(n)} \\
    &\le \frac{(\lambda + 1)^{m(n)}}{f(n)} \;\to\; 0 \qquad \text{as
      $n \to \infty$.}
  \end{split}
  \end{equation*}
\end{proof}
  
\begin{remark}
  \label{rmk:small-support}
  Proposition~\ref{prop:small-support} can be established much more
  easily under the extra assumption that $H$ has sub-exponential word
  growth. Indeed, in this case, one can prove that
  \[
    \lim_{n \to \infty} \frac{\left\vert R_q(n)\right\vert }{\vert B_S(n)\vert } = 0
  \]
  for any non-decreasing unbounded function
  $q \colon \mathbb{N} \to \mathbb{R}_{>1}$ such that $q \in o(n)$;
  the proof is similar to the one of
  Lemma~\ref{lem:small-supp-general} below.

  If we assume that $H$ is finite, it is easy to see that there exists
  $\alpha \in \mathbb{R}_{>0}$ such that
  \[
    \lim_{n \to \infty} \frac{\left\vert R_{q}(n)\right\vert }{\vert B_S(n)\vert }
    = 0 \qquad \text{for $q \colon \mathbb{N} \to \mathbb{R}_{>1}$, $n
      \mapsto 1+\alpha n$.}
 \]
\end{remark}

Next we establish Theorem~\ref{thm:main-density}, using ideas that are
similar to those in the proof of Proposition~\ref{prop:small-support}:
again we work with perturbations of a given element $g$
  in such a manner that the original element can be retrieved easily.
  We begin with some preparations to establish an auxiliary lemma.

Fix a representative function $\mathcal{W}$ which yields for each
element of $G$ an $S$-expression of shortest possible length,
and fix an element $u \in H \smallsetminus \{1\}$.
Consider $g \in N$ with $\mathcal{W}$-itinerary $I = (\iota,\sigma)$,
viz.\ $I_g = (\iota_g,\sigma_g)$.  We put
\[
  \sigma^+ = \sigma_g^+ = \maxit_{\mathcal{W}}(g) \qquad \text{and}
  \qquad \sigma^- = \sigma_g^- = \minit_{\mathcal{W}}(g).
\]
For the time being, we suppose that
\begin{align*}
  k^+ = k_{\mathcal{W},g}^+ %
  & = \min \{ k \mid 0 \le k \le l_S(g) \text{ and } \sigma(k) =
    \sigma^+ \}, \\
  k^- = k_{\mathcal{W},g}^- %
  & = \min \{ k \mid 0 \le k \le l_S(g) \text{ and } \sigma(k) =
    \sigma^- \}
\end{align*}
satisfy $k^+\le k^-$.  We decompose the itinerary for $g$ as
$I = I_1 \ast I_2 \ast I_3$, where $I_1$, $I_2$, $I_3$ have lengths
$k^+$, $k^- - k^+$, $l_S(g)-k^-$; compare with
Lemma~\ref{lem:itindecomposition}.

If $x = x_{\mathcal{W},g}$, $y = y_{\mathcal{W},g}$,
$z = z_{\mathcal{W},g}$ denote the elements corresponding to $I_1$,
$I_2$, $I_3$ then $g = xyz$; observe that the lengths of
$I_1, I_2, I_3$ are automatically minimal, i.e, equal to
$l_S(x), l_S(y), l_S(z)$.  All this is illustrated schematically in
Figure~\ref{fig:g}.  Observe that $I_1$, associated to $x$, `starts'
at $0$ and `ends' at $\sigma^+$, the shifted $I_2$, associated to $y$,
`starts' at $\sigma^+$ and `ends' at $\sigma^-$, and the shifted
$I_3$, associated to $z$, `starts' at $\sigma^-$ and `ends' at $0$.

\begin{figure}[H]
  \centering
  \begin{tikzpicture}[x=0.75pt,y=0.75pt,yscale=-0.95,xscale=0.95]
    \draw (140,100) -- (570,100) ;
    \draw (350,97.5) -- (350,102.5) ;
    \draw (543.89,97.9) .. controls (543.89,93.3) and (541.56,90.9)
    .. (536.89,90.9) -- (444.65,90.9) .. controls (437.98,90.9) and
    (434.65,88.6) .. (434.65,83.93) .. controls (434.65,88.6) and
    (431.32,90.9) .. (424.65,90.9)(427.65,90.9) -- (332.42,90.9)
    .. controls (327.75,90.9) and (325.42,93.3) .. (325.42,97.9) ;
    \draw (403.5,95.75) .. controls (403.49,91.08) and (401.16,88.75)
    .. (396.49,88.76) -- (296.24,88.97) .. controls (289.57,88.99) and
    (286.23,86.67) .. (286.22,82) .. controls (286.23,86.67) and
    (282.91,89.01) .. (276.24,89.02)(279.24,89.01) -- (175.99,89.23)
    .. controls (171.32,89.24) and (168.99,91.57) .. (169,96.24) ;
    \draw (169,104.13) .. controls (169,108.8) and (171.33,111.13)
    .. (176,111.13) -- (341,111.13) .. controls (347.67,111.13) and
    (351,113.46) .. (351,118.13) .. controls (351,113.46) and
    (354.33,111.13) .. (361,111.13)(358,111.13) -- (537,111.13)
    .. controls (541.67,111.13) and (544,108.8) .. (544,104.13) ;
    \draw (168.9,97.5) -- (168.9,102.5) ;
    \draw (543.86,98.21) -- (543.86,103.21) ;

    \draw (346.8,102.8) node [anchor=north west][inner sep=0.75pt]
    [font=\tiny] [align=left] {$\displaystyle 0$};
    \draw (346,122) node [anchor=north west][inner sep=0.75pt]
    [font=\scriptsize] [align=left] {$\displaystyle y$};
    \draw (282,70) node [anchor=north west][inner sep=0.75pt]
    [font=\scriptsize] [align=left] {$\displaystyle z$};
    \draw (430.8,72) node [anchor=north west][inner sep=0.75pt]
    [font=\scriptsize] [align=left] {$\displaystyle x$};
      
    \draw (150,113) node [anchor=north west][inner sep=0.75pt]
    [font=\scriptsize] [align=left]
    {$\displaystyle \sigma^- = \minit(I)$};
    \draw (520,112) node [anchor=north west][inner sep=0.75pt]
    [font=\scriptsize] [align=left]
    {$\displaystyle \sigma^+ = \maxit(I)$};
  \end{tikzpicture}

  \caption{A schematic illustration of the decomposition $g=xyz$.}
  \label{fig:g}
\end{figure}
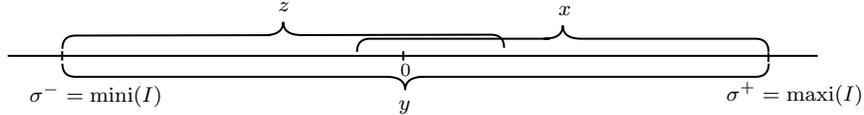

\vspace{5pt}

Next, we put to use the element $u \in H \smallsetminus\{1\}$ that was
fixed and define, for any given
$J \subseteq [\sigma^-,\sigma^+]_\mathbb{Z}$,
perturbations
\[
  \dot x(J) = \dot x_{\mathcal{W},g}(J,u), \qquad \dot y(J) = \dot
  y_{\mathcal{W},g}(J,u), \qquad \dot z(J) = \dot
  z_{\mathcal{W},g}(J,u)
\]
of the elements $x,y,z$ that are specified by
\begin{multline}
  \label{equ:rho-x-y-z}
  \rho(\dot x(J)) = \rho(x) = -\sigma^+, \quad \rho(\dot y(J)) =
  \rho(y) = -\sigma^- + \sigma^+, \\ \rho(\dot z(J)) = \rho(z) =
  \sigma^{-}
\end{multline}
and
\begin{equation}
  \label{equ:def-xyz-dot}
  \begin{split}
    \dot x(J)_{\vert i} & =
    \begin{cases}
      x_{\vert i} \, u & \text{for $i \in J_{\ge 0}$,} \\
      x_{\vert i} & \text{otherwise,}
    \end{cases}
    \\
    \dot y(J)_{\vert i}  & =
    \begin{cases}
      u^{\, -1} \, y_{\vert i} & \text{for $i \in \mathbb{Z}$ such
        that $i +\sigma^+ \in J_{\ge 0}$,} \\
       y_{\vert i} \, u^{\, -1} & \text{for $i \in \mathbb{Z}$ such
        that $i +\sigma^+ \in J_{< 0}$,} \\
       y_{\vert i} & \text{otherwise,}
    \end{cases}
    \\
    \dot z(J)_{\vert i} & =
    \begin{cases}
      u \, z_{\vert i}  & \text{for $i \in \mathbb{Z}$ such that
        $i + \sigma^- \in J_{<0}$,} \\
      z_{\vert i} & \text{otherwise,}
    \end{cases}
  \end{split}
\end{equation}
where we suggestively write $J_{\geq 0} = \{ j \in J \mid
  j \geq 0 \}$ and $J_{<0} = \{ j \in J \mid j < 0 \}$.  We observe that
\begin{equation}
  \label{eq:recover-g}
  g = \dot x(J) \, \dot y(J) \, \dot z(J).
\end{equation}

Let $C = C(S) \in \mathbb{N}$ be as is in
Lemma~\ref{lem:writing}\eqref{enu:lem-1}.   We call
\[
  \ddot g(J) = \dot x(J) \, t^{-2C} \,\dot y(J)^{-1} \, t^{-2C} \dot
  z(J)
\]
the \emph{$J$-variant of $g$}; see Figure~\ref{fig:g(I)}
  for a schematic illustration.

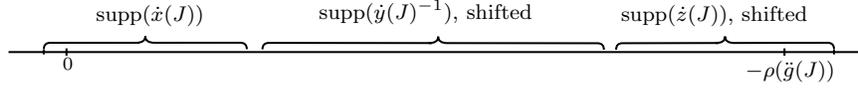
\begin{figure}[H]
  \centering
  \begin{tikzpicture}[x=0.37pt,y=0.37pt,yscale=-0.95,xscale=0.95]

    \draw (29,104) -- (949,104) ;
    \draw (91.4,100) -- (91.4,108) ;
    \draw (864,100) -- (864,108) ;
    \draw (285.42,100.96) .. controls (285.42,96.29) and
    (283.09,93.96) .. (278.42,93.96) -- (186.19,93.99) .. controls
    (179.52,93.99) and (176.19,91.66) .. (176.18,86.99) .. controls
    (176.19,91.66) and (172.86,93.99) .. (166.19,94)(169.19,94) --
    (73.95,94.03) .. controls (69.28,94.03) and (66.95,96.36)
    .. (66.96,101.03) ;
    \draw (917.35,100.98) .. controls (917.34,96.31) and (915,93.99)
    .. (910.33,94) -- (810.08,94.21) .. controls (803.41,94.22) and
    (800.08,91.9) .. (800.07,87.23) .. controls (800.08,91.9) and
    (796.75,94.24) .. (790.08,94.25)(793.08,94.24) -- (689.83,94.46)
    .. controls (685.16,94.47) and (682.84,96.8) .. (682.85,101.47) ;
    \draw (669.68,101.51) .. controls (669.68,96.84) and
    (667.35,94.51) .. (662.68,94.5) -- (501.36,94.33) .. controls
    (494.69,94.32) and (491.36,91.99) .. (491.36,87.32) .. controls
    (491.36,91.99) and (488.03,94.32) .. (481.36,94.31)(484.36,94.31)
    -- (309.27,94.13) .. controls (304.6,94.12) and (302.27,96.45)
    .. (302.27,101.12) ;
    \draw (66.88,101.5) -- (66.88,106.5) ;
    \draw (917.33,101.5) -- (917.33,106.5) ;

    \draw (84,107) node [anchor=north west][inner sep=0.75pt]
    [font=\tiny] [align=left] {$\displaystyle 0$};
    \draw (120,50) node [anchor=north west][inner sep=0.75pt]
    [font=\scriptsize] [align=left]
    {$\displaystyle \supp(\dot x( J))$};
    \draw (360,45) node [anchor=north west][inner sep=0.75pt]
    [font=\scriptsize] [align=left]
    {$\displaystyle \supp(\dot y(J)^{-1})$, shifted};
    \draw (685,50) node [anchor=north west][inner sep=0.75pt]
    [font=\scriptsize] [align=left]
    {$\displaystyle \supp(\dot z(J))$, shifted};
    \draw (819,110) node [anchor=north west][inner sep=0.75pt]
    [font=\scriptsize] [align=left]
    {$\displaystyle -\rho(\ddot g(J))$};
  \end{tikzpicture}
  \caption{A schematic illustration of the support components of
    $\ddot g(J)$.}
  \label{fig:g(I)}
\end{figure}

Observe that
\begin{equation}
  \label{eq:ddot-rho}
  \ddot g(J) \in N t^{\rho(\ddot g(J))}, \qquad \text{where
    $\rho(\ddot g(J)) = 2 ( \sigma_g^- - \sigma_g^+ ) - 4C
    \le -4$.}
\end{equation}
Up to now we assumed that $k^+ \le k^-$. If instead $k^- < k^+$, a
similar construction at this stage yields elements
\begin{equation}
  \label{eq:ddot-rho-minus}
  \ddot g(J) \in Nt^{\rho(\ddot g(J))}, \qquad
  \text{where
    $\rho(\ddot g(J)) = 2 (\sigma_g^+ - \sigma_g^- ) + 4C \ge 4$;}
\end{equation}
in particular, there is no overlap between elements $\ddot g(J)$
arising from these two different cases.

For our purposes, it suffices to work with subsets $J
  \subseteq [\sigma^-,\sigma^+]_\mathbb{Z}$ of size $\vert J \vert =
  2$ and  we streamline the discussion to this situation.

\begin{lemma}
  \label{lem:ddot-properties}
  In the set-up described above, suppose that
  $J \subseteq [ \sigma^-, \sigma^+]_\mathbb{Z}$ with
  $\vert J \vert = 2$.  Let $D = D(S,u) \in \mathbb{N}$ be as in
  Lemma~\ref{lem:writing}\eqref{enu:lem-2}. Then
  \begin{enumerate}[\rm (i)]
  \item\label{enu:ball-ddot} $ l_S(\ddot g(J)) \leq l_S(g) + D' $ for
    $D' = 6D + 2\, l_S \big( t^{2C} \big)$;
  \item \label{enu:recover-main} the element $g$ can be recovered from
    $\ddot g(J)$ and any one of $\sigma^+, \sigma^-$;    
  \item \label{enu:pw-distinct-main} the resulting variants of $g$ are
    pairwise distinct, i.e., $\ddot g(J) \ne \ddot g(J')$ for all
    $ J' \subseteq [ \sigma^-, \sigma^+]_\mathbb{Z}$ with
    $\vert J' \vert =2$ and $J \ne J'$.
  \end{enumerate}
\end{lemma}

\begin{proof}
  \eqref{enu:ball-ddot} Since
  \begin{align*}
    J_{\ge 0} %
    & \subseteq [0,\sigma^+]_\mathbb{Z} \subseteq
      [\minit(I_1),\maxit(I_1)]_\mathbb{Z},\\ 
    J - \sigma^+ %
    & \subseteq [\sigma^- - \sigma^+,0]_\mathbb{Z} =
      [\minit(I_2),\maxit(I_2)]_\mathbb{Z} ,\\  
    J_{<0} - \sigma^- %
    & \subseteq  [0,-\sigma^-]_\mathbb{Z} \subseteq
      [\minit(I_3),\maxit(I_3)]_\mathbb{Z} 
  \end{align*}
  we can apply Lemma~\ref{lem:writing}\eqref{enu:lem-2}, if necessary
  twice, to deduce that
  \[
    l_S( \dot x(J)) \leq l_S(x)+2D, \quad l_S( \dot y(J)) \leq l_S(y)
    +2D, \quad l_S ( \dot z(J)) \leq l_S(z) +2D.
  \]
  Since $l_S(x) + l_S(y) + l_S(z) = l_S(g)$, this gives
  \[
    l_S(\ddot g(J)) \leq l_S(g) + D'  \qquad \text{for
      $D' = 6D + 2 \, l_S \big( t^{2C} \big)$}.
  \]

  \eqref{enu:recover-main} As in the discussion above suppose that
  $k^+ = k^+_{\mathcal{W},g}$ and $k^- = k^-_{\mathcal{W},g}$ satisfy
  $k^+ \le k^-$; the other case $k^- < k^+$ can be dealt with
  similarly.  We have to check that $g$ can be recovered
  from~$\ddot g(J)$, if we are allowed to use one of the parameters
  $\sigma^+, \sigma^-$.  Indeed, from
  $-\rho(\ddot g(J)) = 2 \big( \sigma^+ - \sigma^- \big) +4C$ we
  deduce that in such a case both, $\sigma^+$ and $\sigma^-$ are
  available to us.  Furthermore,
  Lemma~\ref{lem:writing}\eqref{enu:lem-1} gives
  \begin{align*}
    \supp\! \big( \dot x (J) \big) %
    & \subseteq [\sigma^-- C+1, \sigma^++ C-1]_\mathbb{Z}, \\
    \supp \!\big( \dot y(J)^{-1} \big)  %
    &  \subseteq [-C+1,\sigma^+ - \sigma^- +C-1]_\mathbb{Z}, \\
    \supp\!\big(\dot z(J)\big) %
    & \subseteq [-C+1, \sigma^+ -\sigma^-+C-1]_\mathbb{Z},
  \end{align*}
  and thus 
  \begin{multline*}
      \supp \!\big( \ddot g(J) \big)
       = \supp \!\big( \dot x (J) \big) 
      \cupdot \left( \supp \!\big( \dot y(J)^{-1} \big) +
        \sigma^+ + 2 C \right) 
      \\ \cupdot \left( \supp \!\big( \dot z(J) \big) +
        2 \sigma^+ - \sigma^- + 4C \right)
  \end{multline*}
  allows us to recover $\dot x(J)$, $\dot y(J)$ and
  $\dot z(J)$ via \eqref{equ:rho-x-y-z} and
  \begin{align*}
    \dot x(J)_{\vert i} & = %
    \begin{cases}
      \ddot g(J)_{\vert i} & \text{for $i \in [\sigma^-- C, \sigma^++
        C]_\mathbb{Z}$,} \\
      1 & \text{for
        $i \in \mathbb{Z} \smallsetminus [\sigma^-- C, \sigma^++
        C]_\mathbb{Z}$,}
    \end{cases}
    \\
    (\dot y(J)^{-1})_{\vert i} & = %
    \begin{cases}
      \ddot g(J)_{\vert  \, i +
        \sigma^+ + 2C} & \text{for $i \in [-C,
        \sigma^+-\sigma^-+C]_\mathbb{Z}$,} \\ 
      1 & \text{for
        $i \in \mathbb{Z} \smallsetminus [-C,
        \sigma^+-\sigma^-+C]_\mathbb{Z}$,}
    \end{cases}
    \\
    \dot z(J)_{\vert i} & = %
    \begin{cases}
      \ddot g(J)_{\vert  \, i + 2 \sigma^+ - \sigma^- + 4C} & 
      \text{for $i \in [-C, \sigma^+-\sigma^-+C]_\mathbb{Z}$,} \\
      1 & \text{for $i \in \mathbb{Z} \smallsetminus [-C,
        \sigma^+-\sigma^-+C]_\mathbb{Z}$.} 
    \end{cases}
  \end{align*}
  Using \eqref{eq:recover-g}, we recover~$g = \dot x(J) \, \dot y(J)
  \, \dot z(J)$.
  
  \smallskip

  \eqref{enu:pw-distinct-main} Again we suppose that
  $k^+ = k^+_{\mathcal{W},g}$ and $k^- = k^-_{\mathcal{W},g}$ satisfy
  $k^+ \le k^-$; the other case $k^- < k^+$ can be dealt with
  similarly.  Let $J' \subseteq [\sigma^-, \sigma^+]_\mathbb{Z}$
  with $\vert J' \vert = 2$ such that
  $\ddot g(J) = \ddot g(J')$.  As explained above, we can not only
  recover $g$ but even $\dot x(J) = \dot x(J')$,
  $\dot y(J) = \dot y(J')$ and $\dot z(J) = \dot z(J')$ from
  $\ddot g(J) = \ddot g(J')$ and $\sigma^+$, say.  Since $u \ne 1$ we
  deduce from \eqref{equ:def-xyz-dot} that $J = J'$.
\end{proof}

\begin{proof}[Proof of Theorem~\ref{thm:main-density}] We
    continue within the set-up established above; in particular, we
    employ the $J$-variants $\ddot g(J)$ of elements $g \in N$ for
    two-element subsets
    $J \subseteq [\sigma^-_g,\sigma^+_g]_\mathbb{Z}$, with respect to
    a fixed representative function $\mathcal{W}$ and a chosen element
    $u \in H \smallsetminus \{1\}$.

  Let $q \colon \mathbb{N} \to \mathbb{R}_{\ge 1}$ be a non-decreasing
  unbounded function such that $q \in o(\log n)$.  We make use of the
  decomposition
  \begin{equation}\label{equ:A-cap-B-decomp}
    N \cap B_S(n) = R_{q}(n) \cupdot R_{q}^\flat(n),
    \qquad \text{for $n \in \mathbb{N}$,} 
  \end{equation}
  where $R_{q}(n) = R_{\mathcal{W},q}(n)$ is defined as in
  Proposition~\ref{prop:small-support} and
  $R_{q}^\flat(n) = R_{\mathcal{W},q}^\flat(n)$ denotes the
  corresponding complement in $N \cap B_S(n)$.  Let
    $D' \in \mathbb{N}$ be as in
    Lemma~\ref{lem:ddot-properties}\eqref{enu:ball-ddot}.  Below we
    show that
  \begin{equation}
    \label{eq:D'-bound}
    \vert B_S(n + D')\vert  > \frac{q(n)}{2} \, \vert
    R_{q}^\flat(n)\vert  \qquad \text{for  $n \in \mathbb{N}$.}
  \end{equation}
  This bound and submultiplicativity yield
  \[    
    \frac{\vert R_{q}^\flat(n)\vert }{\vert B_S(n)\vert } < \frac{2
      \vert B_S(n + D')\vert }{q(n)\vert B_S(n)\vert } \le \frac{2
      \vert B_S(D')\vert }{q(n)} \;\to\; 0 \quad \text{as
      $n \to \infty$.}
  \]
  Together with Proposition~\ref{prop:small-support} we deduce from
  \eqref{equ:A-cap-B-decomp} that $N$ has density zero:
  \[
    \dens_S(N) = \lim_{n \to \infty}\frac{\vert
      N \cap B_S(n)\vert }{\vert B_S(n)\vert } = 0,
  \]
properly as a limit.

  \medskip
  
It remains to establish \eqref{eq:D'-bound}.  The set
  $R_{q}^\flat(n)$ decomposes into a disjoint union of subsets
  \[
    R_{q,\ell}^\flat(n) = \{ g \in N \cap B_S(n) \mid
    \sigma_g^+ - \sigma_g^- =\ell \}, \quad \ell > q(n),
  \]
  and the map
  \begin{align*}
    F_n \colon R_{q}^\flat(n) %
    & \to \mathcal{P} \!\left( B_S(n + D') \right), \\
    g  & \mapsto \big\{ \ddot g(J) \mid J \subseteq [\sigma_g^-,
         \sigma_g^+]_\mathbb{Z} \text{ with } \vert J\vert  =2 \big\}
  \end{align*}
  restricts for each $\ell \in \mathbb{N}$ with $\ell > q(n)$, to a
  mapping
  \[
    F_{n,\ell} \colon R_{q,\ell}^\flat(n) \to \mathcal{P} \!\left(
      \big( N t^{-2\ell -4C} \cup N
      t^{2\ell+4C} \big) \cap B_S(n + D') \right);
  \]
see Lemma~\ref{lem:ddot-properties}\eqref{enu:ball-ddot},
    \eqref{eq:ddot-rho} and~\eqref{eq:ddot-rho-minus}.
 
    We contend that for every
  $h \in \big( N t^{-2\ell -4C} \cup N
  t^{2\ell +4C} \big) \cap B_S(n + D'),$ where $\ell > q(n)$, there
  are at most $\ell+1$ elements $g \in R_{q,\ell}^\flat(n)$ such that
  $h \in F_n(g)$.  Indeed, suppose that
  $h \in N t^{2\ell +4C} \cap B_S(n + D')$, with
  $\ell > q(n)$, and suppose that $g \in R_{q,\ell}^\flat(n)$ such
  that $h = \ddot g(J)$ for some
  $J \subseteq [\sigma_g^-, \sigma_g^+]_\mathbb{Z}$ with
  $\vert J\vert =2$.  Then $\sigma_g^+ \in [0,\ell]_\mathbb{Z}$ takes
  one of $\ell +1$ values, and once $\sigma^+$ is fixed, there is a
  way of recovering~$g$, by
 Lemma~\ref{lem:ddot-properties}\eqref{enu:recover-main}.  For
  $h \in Nt^{-2\ell -4C} \cap B_S(n + D')$ the argument
  is similar. 
  
  From this observation and
    Lemma~\ref{lem:ddot-properties}\eqref{enu:pw-distinct-main} we
  conclude that
  \begin{multline*}
    \left\vert \big( N t^{-2\ell -4C} \cup N
      t^{2\ell+4C} \big) \cap B_S(n + D') \right\vert \ge
    \frac{1}{\ell+1} \binom{\ell +1}{2} \, \vert
    R_{q,\ell}^\flat(n)\vert \\ > \frac{q(n)}{2} \, \vert
    R_{q,\ell}^\flat(n)\vert .
  \end{multline*}
  Hence
  \[
    \vert B_S(n+D')\vert > \frac{q(n)}{2} \, \sum_{\ell > q(n)}
    \left\vert R_{q,\ell}^\flat(n) \right\vert = \frac{q(n)}{2} \,
    \left\vert R_{q}^\flat(n)\right\vert,
  \]
  which is the bound~\eqref{eq:D'-bound} we aimed for.
\end{proof}


\section{Proof of Theorem \ref{thm:main-2}}

Throughout this section let $G$ denote a finitely generated group of
exponential word growth of the form
$G= N \rtimes \langle t \rangle$, where
\begin{enumerate}[\rm (a)]
\item the subgroup $\langle t \rangle$ is infinite cyclic;
\item the normal subgroup
  $N = \langle \bigcup \big\{ H^{t^i} \mid i \in \mathbb{Z} \big\}
  \rangle$ is generated by the $\langle t \rangle$-conjugates of a
  finitely generated subgroup~$H$ $N$;
\item \label{enu:conj-commute} the $\langle t \rangle$-conjugates of
  this group $H$ commute elementwise:
  $\big[H^{t^i}, H^{t^j} \big] = 1$ for all $i, j \in \mathbb{Z}$ with
  $H^{t^i} \ne H^{t^j}$.
\end{enumerate}
Suppose further that
$S_0 = \{a_1, \dots, a_d\} \subseteq H$ is a finite
symmetric generating set for $H$ and that the exponential growth rates
of $H$ with respect to $S_0$ and of $G$ with respect to
$S = S_0 \cup \{ t, t^{-1} \}$ satisfy
\begin{equation}
  \label{eq:inequality-later}
  \lim_{n \rightarrow \infty} \sqrt[n]{\vert B_{H,S_0}(n)\vert } < \lim_{n
    \rightarrow \infty} \sqrt[n]{\vert B_{G,S}(n)\vert }. 
\end{equation}
This is essentially the setting of Theorem~\ref{thm:main-2}; for
technical reasons we prefer to work with symmetric generating sets.
Our ultimate aim is to show that $\dens_S(N)=0$.

Using the commutation rules recorded in~\eqref{enu:conj-commute}, it is not
difficult to see that every $g \in N$ admits $S$-expressions of
minimal length that take the special form
\begin{align}
  \label{eq:minimal-1}
  g & = t^{-\sigma^-} \cdot \bigg( \prod_{i=\sigma^-}^{\sigma^+-1} 
      \big( w_i(a_1, \dots, a_d) \, t^{-1} \big) \bigg) \cdot
      w_{\sigma^+}(a_1,\ldots,a_d) \cdot t^{\sigma^+}, \\
  \label{eq:minimal-2}
  g & = t^{-\sigma^+} \cdot \bigg(  \prod_{j=\sigma^-}^{\sigma^+ -1} 
      \big( w_{\sigma^+ + \sigma^- - j}(a_1, \dots, a_d) \, t \big) \bigg) \cdot
      w_{\sigma^-}(a_1,\ldots,a_d) \cdot t^{\sigma^-},
\end{align}

where the parameters $\sigma^-, \sigma^+ \in \mathbb{Z}$ satisfy
$\sigma^- \le \sigma^+$ and, for every
$i \in [\sigma^-,\sigma^+]_\mathbb{Z}$, we have picked a suitable
semigroup word $w_i = w_i(Y_1,\ldots,Y_d)$ in $d$ variables of length
$l_{S_0}(w_i(a_1,\ldots,a_d))$.  The lengths of the
expressions~\eqref{eq:minimal-1} and \eqref{eq:minimal-2} are equal to
\[
  l_S(g) = \vert \sigma^-\vert  + (\sigma^+ - \sigma^-) + \vert \sigma^+\vert  +
  \sum_{i=\sigma^-}^{\sigma^+} l_{S_0} \!\big(w_i(a_1,\ldots,a_d)
  \big).
\]

For the following we fix, for each $g \in N$, expressions
as described and we use subscripts to stress the dependency on~$g$: we
write $\sigma_g^-$, $\sigma_g^+$ and $w_{g,i}$ for
$i \in [\sigma_g^-,\sigma_g^+]_\mathbb{Z}$, where necessary.  The
notation is meant to be reminiscent of the one introduced in
Definition~\ref{def:itinerary}, but one needs to keep in mind that we
are dealing with a larger class of groups now.

\begin{lemma}
  \label{lem:small-supp-general}
In addition to the general set-up described above, let
  $q \colon \mathbb{N} \to \mathbb{R}_{>0}$ be a non-decreasing
  unbounded function such that $q \in o(n)$.  Then the
    sequence of sets 
  \[
    R_q(n) = \{ g \in N \cap B_S(n) \mid -q(n) \le \sigma_g^- \le
    \sigma_g^+ \le q(n) \},
  \]
  indexed by $n \in \mathbb{N}$, satisfies
  \[
    \lim_{n \to \infty} \frac{\vert R_q(n)\vert }{\vert B_S(n)\vert } = 0.
  \]
\end{lemma}

\begin{proof}
  For short we set $\mu = \lim_{n \to \infty}\sqrt[n]{\vert B_{H,S_0}(n)\vert }$
  and $\lambda = \lim_{n \to \infty} \sqrt[n]{\vert B_{G,S}(n)\vert }$.
  According to \eqref{eq:inequality-later} we find
  $\varepsilon \in \mathbb{R}_{> 0}$ such that
  $(\mu + \varepsilon)/\lambda \le 1 - \varepsilon$ and
  $M = M_\varepsilon \in \mathbb{N}$ such that
  \[
    \vert B_{H,S_0}(n)\vert  \le M (\mu+\varepsilon)^n \quad \text{for all
      $n \in \mathbb{N}_0$.}
  \]
  This allows us to bound the number of possibilities for the elements
  $w_{g,i}(a_1,\ldots,a_d)$ in an $S$-expression of the
  form~\eqref{eq:minimal-1} for $g \in R_q(n)$ and, writing
  $\tilde q(n) = 2 \lfloor q(n) \rfloor +1$, we obtain
  \begin{align*}
    \vert R_q(n)\vert %
    & \leq  \sum_{\substack{m_{-\lfloor q(n) \rfloor}, \ldots,
      m_{\lfloor q(n) \rfloor} \in \mathbb{N}_0 \text{ st}\\
    m_{-\lfloor q(n) \rfloor} + \ldots
    + m_{\lfloor q(n) \rfloor} \le n}} \;\; \prod_{i =
    -\lfloor q(n) \rfloor}^{\lfloor q(n) \rfloor} \vert
    B_{H,S_0}(m_i)\vert   \\
    & \leq \binom{n + \tilde q(n)}{\tilde q(n)} M^{\tilde q(n)} (\mu+\varepsilon)^n,
  \end{align*}
  and hence
  \begin{equation}
    \label{eq:polynomial}
    \frac{\vert R_q(n)\vert }{\vert B_S(n)\vert } \le \frac{\vert R_q(n)\vert }{\lambda^n}
    \le \binom{n+\tilde q(n)}{\tilde q(n)} M^{\tilde q(n)}
    (1-\varepsilon)^n \quad \text{for $n \in \mathbb{N}$.}
  \end{equation}
  We notice that $q \in o(n)$ implies $\tilde q \in o(n)$.  Thus
  Lemma~\ref{lem:stirling} implies that
  $\binom{n + \tilde q(n)}{\tilde q(n)} M^{\tilde q(n)}$ grows
  sub-exponentially, and the term on the right-hand side
  of \eqref{eq:polynomial} tends to $0$ as $n$ tends to infinity.
\end{proof}

\begin{proof}[Proof of Theorem \ref{thm:main-2}]
  We continue to work in the notational set-up introduced above.  In
  addition we fix a non-decreasing unbounded function
  $q \colon \mathbb{N} \to \mathbb{R}_{\ge 0}$ such that $q \in o(n)$
  and
    \begin{equation}\label{eq:from-lem-exists-q}
      \frac{\vert B_S(n)\vert }{\vert B_S(n-q(n))\vert } \to \infty
      \qquad \text{as $n \to \infty$};
    \end{equation}
    see Proposition~\ref{pro:exists-q}.  As in the proof of
  Theorem~\ref{thm:main-density}, we make use of a decomposition
  \begin{equation*}
    N \cap B_S(n) = R_{q}(n) \cupdot R_{q}^\flat(n),
    \qquad \text{for $n \in \mathbb{N}$,} 
  \end{equation*}
  where $R_{q}(n)$ is defined as in Lemma~\ref{lem:small-supp-general}
  and $R_{q}^\flat(n)$ denotes the corresponding complement in $N \cap B_S(n)$.

  In view of Lemma~\ref{lem:small-supp-general} it suffices to show
  that
  \begin{equation}
    \label{equ:R-flat-0}
    \frac{\vert R_{q}^\flat(n)\vert }{\vert B_S(n)\vert } \to 0 \quad \text{as $n \to
      \infty$.}
  \end{equation}
 It is enough to consider sufficiently large $n$ so that
    $n>q(n)$ holds.  For every such $n$ and $g \in R_{q}^\flat(n)$,
  with chosen minimal $S$-expressions~\eqref{eq:minimal-1} and
  \eqref{eq:minimal-2}, we have $\sigma^- = \sigma_g^- < - q(n)$ or
  $\sigma^+ = \sigma_g^+ > q(n)$, hence
  \[
    \left\{ g t^{-q(n)},  g t^{q(n)} \right\} \cap B_S(n-q(n)) \ne \varnothing.
  \]
  As each of the right translation maps $g \mapsto g t^{-q(n)}$ and
  $g \mapsto g t^{q(n)}$ is injective, we conclude that
  \[
    \vert R_{q}^\flat(n)\vert  \le 2 \vert B_S(n-q(n))\vert, 
  \]
  and thus  \eqref{equ:R-flat-0} follows
    from~\eqref{eq:from-lem-exists-q}.
\end{proof}

\medskip





\end{document}